\documentclass[12pt,reqno]{amsart}
\usepackage{amsfonts}
\usepackage{amssymb,mathrsfs,color}
 \usepackage{color}
\def\norm#1{\|#1\|}

\newcommand{\les}{{\lesssim}}
\newcommand{\ges}{{\gtrsim}}

\newcommand{\F}{\mathcal{F}}

\newcommand{\cS}{\mathcal{S}}

\newcommand{\cP}{\mathcal{P}}
\newcommand{\cU}{\mathcal{U}}
\newcommand{\cN}{\mathcal{N}}
\newcommand{\cV}{\mathcal{V}}
\newcommand{\cW}{\mathcal{W}}

\newcommand{\K}{\mathcal{K}}

\newcommand{\N}{\mathbb{N}}
\newcommand{\R}{\mathbb{R}}
\newcommand{\Z}{\mathbb{Z}}

\newcommand{\al}{\alpha}
\newcommand{\be}{\beta}

\newcommand{\de}{\delta}
\newcommand{\e}{\varepsilon}
\newcommand{\fy}{\varphi}
\newcommand{\om}{\omega}
\newcommand{\la}{\lambda}
\newcommand{\te}{\theta}

\newcommand{\ka}{\kappa}
\newcommand{\x}{\xi}
\newcommand{\y}{\eta}

\newcommand{\ft}{{\mathcal{F}}}
\newcommand{\wh}{\widehat}

\newcommand{\De}{\Delta}
\newcommand{\Om}{\Omega}

\newcommand{\p}{\partial}
\newcommand{\na}{\nabla}

\newcommand{\im}{\mathop{\mathrm{Im}}}

\newcommand{\supp}{\operatorname{supp}}

\newcommand{\lec}{\lesssim}

\newcommand{\I}{\infty}

\newcommand{\fr}{\frac}

\newcommand{\ti}{\widetilde}
\newcommand{\ba}{\overline}

\newcommand{\LR}[1]{{\langle #1 \rangle}}
\newcommand{\lp}[1]{{\bf #1}}
\newcommand{\pe}[1]{#1^{\operatorname{>J}}}
\newcommand{\np}[1]{\mathfrak{#1}}

\newcommand{\tand}{\text{ and }}

\newcommand{\EQ}[1]{\begin{equation}\begin{split} #1 \end{split}\end{equation}}
\newcommand{\Del}[1]{}
\newcommand{\CAS}[1]{\begin{cases} #1 \end{cases}}

\newcommand{\pt}{&}
\newcommand{\pr}{\\ &}
\newcommand{\pq}{\quad}
\newcommand{\pn}{}

\numberwithin{equation}{section}

\newtheorem{thm}{Theorem}[section]
\newtheorem{cor}[thm]{Corollary}
\newtheorem{lem}[thm]{Lemma}

\theoremstyle{remark}
\newtheorem{rem}{Remark}

\begin{document}

\title[Global dynamics below ground state for Klein-Gordon-Zakharov]{Global dynamics below the ground state energy for the Klein-Gordon-Zakharov system
in the 3D radial case}
\author[Z. Guo, K. Nakanishi, S. Wang]{Zihua Guo, Kenji Nakanishi, Shuxia Wang}

\address{LMAM, School of Mathematical Sciences, Peking
University, Beijing 100871, China}
\email{zihuaguo@math.pku.edu.cn}

\address{Department of Mathematics, Kyoto University, Kyoto 606-8502,
Japan}

\email{n-kenji@math.kyoto-u.ac.jp}

\address{School of Mathematical Sciences, Peking
University, Beijing 100871, China}

\email{wangshuxia@pku.edu.cn}

\begin{abstract}
We consider the global dynamics below the ground state energy for the Klein-Gordon-Zakharov system in the 3D radial case; and obtain the dichotomy between scattering and finite time blow up.

\end{abstract}

\maketitle

\section{Introduction}
In this paper, we continue our study \cite{GNW} on the global Cauchy problem for the 3D Klein-Gordon-Zakharov system \EQ{\label{eq:KGZ1}
 \CAS{ \ddot{u} - \De u +u= nu,\\
   \ddot n/\al^2 - \De n = -\De u^2,}}
with the initial data
\begin{align}
u(0,x)=u_0,\,\dot{u}(0,x)=u_1,\  n(0,x)=n_0,\,\dot n(0,x)=n_1,
\end{align}
where $(u,n)(t,x):\R^{1+3}\to\R\times\R$, and $\al>0, \al\neq 1$ denotes the
ion sound speed. It preserves  the energy \EQ{
 E(u,\dot{u},n,\dot{n})=\int_{\R^3}\frac{|u|^2+|\na u|^2+|\dot{u}|^2}{2}+\frac{|D^{-1}\dot n|^2/\al^2+|n|^2}{4}-\frac{n|u|^2}{2} dx,}
where $D:=\sqrt{-\De}$, as well as the radial symmetry.

This system describes the interaction between Langmuir waves and ion
sound waves in a plasma (see \cite{B}, \cite{D}). The local
well-posedness (for arbitrary initial data) and global
well-posedness (for small initial data) of \eqref{eq:KGZ1} with
$\al<1$ in the energy space $H^1\times L^2$ was proved by Ozawa,
Tsutaya and Tsutsumi in \cite{OTT}. We point out that
\eqref{eq:KGZ1} does not have null form structure as in Klainerman
and Machedon \cite{SKMM} and this suggests that when $\alpha=1$ the
system \eqref{eq:KGZ1} may be locally ill-posed in $H^1\times L^2$
(cf. the counter example of Lindblad \cite{L} for similar
equations). Hence, we suppose $\al\neq 1$ here. When the first
equation of \eqref{eq:KGZ1} is replaced by $c^{-2}\ddot{u} - \De u
+c^2u= -nu$, Masmoudi and Nakanishi
studied the limit system ($c, \alpha\rightarrow\I$) and the behavior of their solutions in a
series of papers \cite{MN1}-\cite{MN3}. The instability of standing
wave of Klein-Gordon-Zakharov system was studied in \cite{GGZ},
\cite{GZ} and \cite{OT}. Recently, in \cite{GNW} the authors
obtained scattering for radial initial data with small energy in the
3D case, by using the normal form reduction and radial-improved
Strichartz estimates. The purpose of this paper is to consider the
global dynamics for larger data under the radial symmetry.
The idea of this paper is the same with \cite{GNW2}, in which we studied
the global dynamics of Zakharov system. The main difference is that we can prove blow-up in finite time on one side of the dichotomy of global dynamics, whereas for the Zakharov system the existence of any blow-up solution is still an open problem in three dimensions.

It is well known \cite{BL,C,S} that there exists a unique radial
positive ground state $Q(x)$, solving the static equation \EQ{
\label{stat-KG}
 -\De Q + Q = Q^3,}
with the least energy \EQ{
 J(Q):=\int_{\R^3}\frac{|Q|^2+|\na Q|^2}{2}-\frac{|Q|^4}{4}dx>0,}
among all nontrivial solutions of \eqref{stat-KG}.

Since the Klein-Gordon-Zakharov system \eqref{eq:KGZ1} has the following
radial standing waves \EQ{\label{eq:KGZstandingwave}
 (u,n)=(\pm Q,Q^2),}
the goal of this study is to determine global dynamics of all the
radial solutions $"below"$ the above family of special solutions, in
the spirit of Kenig-Merle \cite{KM}, namely the variational
dichotomy into the scattering solutions and the blowup solutions.
For the dichotomy, we need
to introduce two functionals (for Klein-Gordon equation), both of which are the scaling
derivative of the static Klein-Gordon energy $J$: \EQ{\label{functional}
&K_0(\fy):=\p_{\la}|_{\la=1}J(\la\fy(x))
 =\int_{\R^3}\left(|\fy|^2+|\na\fy|^2-|\fy|^4\right)dx,\\
 &K_2(\fy):=\p_{\la}|_{\la=1}J(\la^{3/2}\fy(\la x))
 =\int_{\R^3}\left(|\na\fy|^2-\frac{3|\fy|^4}{4}\right)dx.}
 The main result of this paper is

\begin{thm}\label{thm}
Assume that \EQ{\label{eq:indata}
 (u_0,u_1,n_0,n_1)\in H^1_r(\R^3)\times L^2_r(\R^3)\times L^2_r(\R^3)\times \dot H^{-1}_r(\R^3)}
is radial and
satisfies
\begin{align}\label{eq:Ethm}
  E(u_0,u_1,n_0,n_1)<J(Q).
\end{align}
Then for both $i=0,2$, we have

(a) if $K_i(u_0)\ge 0$, then \eqref{eq:KGZ1} has a unique global solution $(u,n)$, which scatters both as $t\to\I$ and as $t\to-\I$
in the energy space;

(b) if $K_i(u_0)<0$, then the solution $(u,n)$ of \eqref{eq:KGZ1} blows up in  finite  time.
\end{thm}
\begin{rem}
 The condition \eqref{eq:Ethm} is sharp in view of the standing
wave solutions \eqref{eq:KGZstandingwave}which satisfies $E=J(Q)$ and $K_i=0$ with different behavior from $(a)$ and $(b)$.
\end{rem}
\begin{rem}
The result $(b)$ is also true for non-radial case. See the proof in Section \ref{sec:blow up}.
\end{rem}

\section{Hamiltonian and variational structures}

\subsection{Virial identity}
We derive a virial identity on $\R^d$ here, which is similar to that in \cite{GNW2} and will play a crucial role in the proof of
the scattering.

Let
\EQ{I(t)=-2\LR{\dot{u},(x\cdot\nabla+\frac{d}{2})u}-\frac{1}{\alpha^2}\LR{D^{-1}\dot{n},D^{-1}(x\cdot\nabla+\frac{d+1}{2})n},}
by integration by parts we have
\EQ{I'(t)&=2\|\nabla u\|_{L^2}^2+\frac{1}{2}\norm{n}_{L^2}^2+\frac{1}{2\alpha^2}\norm{\dot{n}}_{\dot{H}^{-1}}^2
-\frac{d+1}{2}\int_{\R^d}nu^2dx\\
&=2K_2(u)+\frac{1}{2\alpha^2}\norm{\dot{n}}_{\dot{H}^{-1}}^2+\frac{1}{2}\norm{n-u^2}_{L^2}^2-\frac{d-1}{2}\LR{n-u^2,u^2}.}

\subsection{Variational estimates}
Let
\EQ{G_0(\fy):=J-\frac{K_0}{4}=\frac{1}{4}\|\fy\|^2_{H^1}, \ \ \ G_2(\fy):=J-\frac{K_2}{3}=\frac{1}{6}\|\nabla\fy\|^2_{L^2}+\frac{1}{2}\|\fy\|^2_{L^2},}
The following characterization of the ground state $Q$ is well known (cf. Lemma 2.1 in \cite{NS}).
\begin{lem}
For $i=0,2$,
\EQ{\label{infima}J(Q)&=\inf\{J(\fy) | 0\neq \fy\in H^1, K_i=0\}\\
&=\inf\{G_i(\fy) | 0\neq \fy\in H^1, K_i\leq0\},}
and these infima are achieved uniquely by the ground states $\pm Q$.
\end{lem}
Since \EQ{
 E(u,\dot{u},n,\dot{n})&=\int_{\R^3}\frac{|u|^2+|\na u|^2+|\dot{u}|^2}{2}+\frac{|D^{-1}\dot{n}|^2/\al^2+|n|^2}{4}-\frac{n|u|^2}{2} dx\\
 &=J(u)+\int_{\R^3}\frac{|\dot{u}|^2}{2}+\frac{|D^{-1}\dot{n}|^2}{4\al^2}+\frac{|n-u^2|^2}{4} dx\geq J(u),}
by energy conservation and Lemma 2.12 in \cite{IMN}, we have
\begin{lem}\label{lem:sign}
Assume that $(u,n)$ is a solution  to \eqref{eq:KGZ1} with maximal
interval $I$ satisfying
\begin{align}
  E(u_0,u_1,n_0,n_1)<J(Q).
\end{align}
Then there exist $\delta>0$ such that for all $t\in I$, one has either
\begin{align}\begin{cases}
 K_0(u)\leq-2(J(Q)-E(u_0,u_1,n_0,n_1)),\\K_2(u)\leq-2(J(Q)-E(u_0,u_1,n_0,n_1)),\end{cases}
\end{align}
or
\begin{align}\begin{cases}
 K_0(u)\geq\min(2(J(Q)-E(u_0,u_1,n_0,n_1)), \delta \|u\|^2_{H^1}),\\ K_2(u)\geq\min(2(J(Q)-E(u_0,u_1,n_0,n_1)), \delta \|\nabla u\|^2_{L^2}). \end{cases}
\end{align}
Especially, $K_0(u)$ and $K_2(u)$ have the same sign
and neither of them changes the sign on $I$.
\end{lem}

\begin{cor}\label{cor:cor}
Assume that $(u,n)$ is a solution to \eqref{eq:KGZ1} with maximal
interval $I$ satisfying
\begin{align}
  E(u_0,u_1,n_0,n_1)<J(Q);\ \ \ K_i(u_0)\geq0\  \text{for}\ i=0,2.
\end{align}
Then $I=(-\I,\I)$, and
moreover,
\begin{align}\label{eq:Eequiv}
 E(u,\dot{u},n,\dot{n})&\sim \norm{u}_{H^1}^2+\norm{\dot{u}}_{L^2}^2+\norm{n}_{L^2}^2+\norm{\dot{n}}_{\dot{H}^{-1}}^2\\
 &\sim
\norm{u_0}_{H^1}^2+\norm{u_1}_{L^2}^2+\norm{n_0}_{L^2}^2+\norm{n_1}_{\dot{H}^{-1}}^2.
\end{align}
\end{cor}

\begin{proof}
Since $K_2(u(t))\geq0$, we
get \eqref{eq:Eequiv} immediately from \EQ{\label{energy est}
J(Q)\geq& E(u,\dot{u},n,\dot{n})-K_2(u(t))/3\\
=&\frac{1}{6}\|\na u\|_{L^2}^2+\fr
{\|u\|_{L^2}^2+\|\dot{u}\|_{L^2}^2}2+\frac{\|D^{-1}\dot{n}\|_{L^2}^2}{4\al^2}+\frac
14\|n-u^2\|_{L^2}^2. } By the Sobolev inequality $\norm{u}_{L^4}\les
\norm{u}_{H^1}$, $(u,\dot{u},n,\dot{n})(t)$ is bounded in $H^1\times
L^2\times L^2\times\dot{H}^{-1}$, and thus by the local
wellposedness we have $I=(-\I,\I)$.
\end{proof}

So far, the global well-posedness of part (a) of Theorem \ref{thm} is
proved. It remains to prove the scattering and part (b).
 For all
purposes, the virial estimates play crucial roles and we will use the following key observation.

\begin{lem}\label{lem:var}
Let $\fy\in H^1(\R^3)$ and $\nu\ge 0$ satisfy \EQ{
 J(\fy)+\frac{\nu^2}{4} \le J(Q).}
Then, for $i=0,2$,  we have \EQ{
 \CAS{K_i(\fy) \ge 0 \implies 4K_2(\fy) + \nu^2 \ge \sqrt{6}\nu\|\fy\|_{L^4}^2,\\
 K_i(\fy)\le 0 \implies K_0(\fy)  \le -\nu\|\fy\|_{L^4}^2.}}
\end{lem}
\begin{proof}
The first inequality has been proved in \cite{GNW2} by considering the $L^2$ scaling of $\fy$.
Here, we only prove the second one and the third one.

If $K_i(\fy)=0$ then $\nu=0$ and the conclusion is trivial.
Hence we may assume $K_i(\fy)\not=0$ as well as $\nu>0$.

For $K_i(\fy)<0$, consider the $L^0$ scaling of $\fy$. Since \EQ{
 \pt J(\mu\fy)=\frac{\mu^2}{2}\|\fy\|_{H^1}^2-\frac{\mu^4}{4}\|\fy\|_{L^4}^4,
 \pr \mu\p_\mu J(\mu\fy)=K_0(\mu\fy)=\mu^2\|\fy\|_{H^1}^2-\mu^4\|\fy\|_{L^4}^4.}
There is a unique $0<\mu\not=1$ such that \EQ{ \label{eq mu}
 \|\fy\|_{H^1}^2=\mu^2\|\fy\|_{L^4}^4,}
which is equivalent to $K_0(\mu\fy)=0$. Then the variational
characterization of $Q$ implies $J(\mu\fy)\ge J(Q)$, and so \EQ{
 \frac{\nu^2}{4} \le J(\mu\fy)-J(\fy)
  =\frac{(\mu^2-1)^2}{4}\|\fy\|_4^4.}
By denoting $X:=\|\fy\|_4^2/\nu$, the above inequality is rewritten as \EQ{
\label{X lowbd}
 |\mu^2-1|X \ge 1.}
Hence, for $K_0(\fy)<0$, or equivalently $0<\mu<1$, \EQ{
 \frac{\nu\|\fy\|_4^2}{-K_0(\fy)}&=\fr 1{(1-\mu^2)X}<1.}
Therefore, the proof of the lemma is completed.
\end{proof}

\section{Blow up in finite time}\label{sec:blow up}
This section is devoted to proving part (b) of Theorem \ref{thm}.
Suppose for contradiction that the solution $(u,n)$ exists for all $t>0$.
We
define an auxiliary function
\EQ{\tilde{I}_1(t)=\|u(t)\|^2_{L^2}.}
By direct calculation,
\EQ{\tilde{I}_1''(t)&=2\|\dot{u}\|_{L^2}^2-2\|\nabla
u\|_{L^2}^2
-2\| u\|_{L^2}^2+2\int_{\R^3}nu^2dx\\
&=2\norm{\dot{u}}_{L^2}^2-2K_0(u)+2\langle
n-u^2,u^2\rangle.}
By Lemma \ref{lem:var} with $\nu>\|n-u^2\|_{L^2}$ and H\"older, we have \EQ{\label{I
convex}\tilde{I}_1''(t)\geq\kappa,} for some $\kappa\sim
J(Q)-E(u_0,u_1,n_0,n_1)>0$. Hence $\tilde{I}_1(t)$ is a uniformly convex
function of $t$, and $\tilde{I}_1(t)\to \infty$ as $t\to +\infty$.

Let $\psi\in C_0^\I(\R^3)$ be a fixed radial function satisfying $0\le\psi\le 1$,
$\psi(x)=1$ for $|x|\le 1$ and $\psi(x)=0$ for $|x|\ge 2$. For any $n_0\in L^2$, we can choose $K>0$ large enough such that
$\norm{\ft^{-1}\psi(2^K\xi)\ft n_0}_{L^2}=\e_0\ll 1$, where $\e_0$ will be decided later. Denote $n_{0,K}=P_{\leq -K}n_0:=\ft^{-1}\psi(2^K\xi)\ft n_0$. Inspired by \cite{GZ} and \cite{OT}, we set
\EQ{\tilde{I}_2(t)=\|u(t)\|^2_{L^2}+\frac{1}{2\alpha^2}\|n(t)-n_{0,K}\|^2_{\dot{H}^{-1}},}
then \EQ{\tilde{I}_2'(t)=2\langle u(t),
\dot{u}(t)\rangle_{L^2}+\frac{1}{\alpha^2}\langle n(t)-n_{0,K},
\dot{n}(t)\rangle_{\dot{H}^{-1}}}
and for $t$ sufficiently large, we have
\EQ{\tilde{I}_2''(t)&=5\norm{\dot{u}}_{L^2}^2+\frac{5}{2\alpha^2}\norm{\dot{n}}_{\dot{H}^{-1}}^2-6E(u,\dot{u},n,\dot{n})\\
&\quad+\|\nabla
u\|_{L^2}^2+
\|u\|_{L^2}^2+\frac{1}{2}\norm{n}_{L^2}^2+\langle n_{0,K},
n-u^2\rangle\\
&\geq5(\norm{\dot{u}}_{L^2}^2+\frac{1}{2\alpha^2}\norm{\dot{n}}_{\dot{H}^{-1}}^2),}
since $|\langle n_{0,K},
n-u^2\rangle|\leq \norm{n_{0,K}}_{L^2}\norm{n-u^2}_{L^2}\leq C\e_0(\norm{n}_{L^2}+\norm{\nabla u}_{L^2}^2+\norm{u}_{L^2}^2)$, and by choosing $\e_0\ll 1$ such that $C\e_0\leq 1/100$. Thus we get
\EQ{\tilde{I}_2(t)\tilde{I}_2''(t)\geq\frac{5}{4}(\tilde{I}_2'(t))^2.}
Using this result, we can obtain
\EQ{[\tilde{I}_2^{-1/4}(t)]''=-\frac{1}{4}\tilde{I}_2^{-9/4}(t)[\tilde{I}_2(t)\tilde{I}_2''(t)-\frac{5}{4}(\tilde{I}_2'(t))^2]\leq0.}
Hence $\tilde{I}_2^{-1/4}(t)$ is concave for sufficiently large $t$; and there exists a finite time $T^*$ such that $\lim_{t\rightarrow T^* }\tilde{I}_2(t)=\infty$.
Since \EQ{
n=\cos(\al Dt)n_0+\frac{\sin(\al Dt)}{\al D}n_1+\im\alpha\int_0^te^{i\al D(t-s)}D|u(s)|^2ds,}
 we have
\EQ{&\|D^{-1}(n-n_{0,K})\|_{L_t^\infty L_x^2((0,T^*)\times \R^3)}\\
\lesssim& \|D^{-1}(I-P_{\leq -K})n_0\|_{L_x^2(\R^3)}+\alpha T^*\norm{P_{\leq -K}n_0}_2\\
&+T^*\|D^{-1}n_1\|_{L_x^2(\R^3)}+T^*\|u\|_{L_t^\infty L_x^4((0,T^*)\times \R^3)}\\
\lesssim& (2^K+T^*)\|n_0\|_{L_x^2( \R^3)}+T^*\|D^{-1}n_1\|_{L_x^2(\R^3)}+T^*\|u\|_{L_t^\infty
H_x^1((0,T^*)\times \R^3)}.}
Thus one has $T<\I$ such that
\EQ{\limsup_{t\rightarrow T^-}\|u(t)\|^2_{H^1}=\I.}

\section{Concentration-compactness procedure}
It remains to prove the scattering in part (a) of Theorem \ref{thm}.
Thanks to the variational estimates in Section 2, we can proceed as
Kenig-Merle.

To simply the presentation, we rewrite the system \eqref{eq:KGZ1}
into the first order as usual. Let \EQ{ \cU:=u - i\LR D^{-1}\dot
u,\quad  \cN:=n - iD^{-1}\dot n/\al,} where $\LR
D=(I-\triangle)^{1/2}$, then the equations for $(\cU,\cN)$ are
\EQ{\label{eq:KGZ2}
 \CAS{ (i\p_t+\LR D)\cU =\LR D^{-1}(\cN\cU/4+\bar{\cN}\cU/4+
 \cN\bar{\cU}/4+\bar{\cN}\bar{\cU}/4),\\
   (i\p_t+\al D) \cN = \al D(\cU\bar{\cU}/4+\bar{\cU}\cU/4+\cU^2/4+\bar{\cU}^2/4)}}
with initial data $(\cU_0,\cN_0)\in H^1\times L^2$ and energy
\EQ{E(\cU,\cN):=E(u,\dot{u},n,\dot{n}),\ \ \ K_i(\cU):=K_i(u) \ \text{for }i=0,2.}

For each $0\le a\le J(Q)$ and $\la>0$, let
\EQ{&\K^+(a):=\{(f,g)\in H^1_r\times L^2_r\mid E(f,g)<a,\ K_i(f)\ge0,i=0,2\}\\
&\cS(a):=\sup\{\|(\cU,\cN)\|_S \mid (\cU(0),\cN(0))\in\K^+(a),\ \text{$(\cU,\cN)$ sol.}\},}
where $S$ denotes a norm containing almost all the Strichartz norms for radial free solutions, including $L^\I_t(H^1\times L^2)$. See \eqref{def S} for the precise definition. For any time interval $I$, we will denote by $S(I)$ the restriction of $S$ onto $I$.

From Corollary \ref{cor:cor} we already know that all solutions
starting from $\K^+(a)$ stays there globally in time. What we
want to prove is the uniform scattering below the ground state
energy, i.e. $\cS(a)<\I$ for all $a<J(Q)$.
Let \EQ{
 E^* :=\sup\{a>0 \mid \cS(a)<\I\}.}
The small data scattering in \cite{GNW} implies that $E^*>0$, and
the existence of the ground state soliton implies that $E^*\le
J(Q)$. We will prove  $E^*= J(Q)$ by contradiction, and thus
finish the proof of Theorem \ref{thm} (a). The main result in this
section is
\begin{lem}[Existence of critical element]\label{lem:crit}
Suppose $E^*< J(Q)$, then there is a global solution $(\cU,\cN)$ in
$\K^+(a)$ satisfying
 \EQ{
 E(\cU,\cN)=E^*,\ \ \ \|(\cU,\cN)\|_{S(-\I,0)}=\|(\cU,\cN)\|_{S(0,\I)}=\I.
 }
 Moreover, $\{(\cU,\cN)(t)\ |\ t\in\mathbb{R}\}$ is precompact in $H^1_x\times L_x^2$.
\end{lem}

We will prove this lemma by the following concentration-compactness
procedure. The main difference from  Klein-Gordon is that we need to
work with the solutions after the normal form transform. In
particular, we have some nonlinear terms without time integration
(or the Duhamel form). Besides that, we have various different
interactions, for which we need to use different norms or exponents.

\subsection{Profiles for the radial Klein-Gordon-Zakharov}
First we recall the free profile decomposition of Bahouri-G\'erard
type \cite{BG}. Actually we do not need its full power, as we can
freeze scaling and space positions of the profiles thanks to the
radial symmetry and the regularity room of our problem. Hence the
setting is essentially the same as the Klein-Gordon case \cite{NS}.
\begin{lem} \label{free prof}
For any bounded sequence $(f_n,g_n)$ in $H^1_r\times L^2_r$, there
is a subsequence $(f_n',g_n')$, $\bar{J}\in \N\cup \{\infty\}$, a
bounded sequence $\{\lp f^j,\lp g^j\}_{1\le j< \bar{J}}$ in
$H^1_r\times L^2_r$, and sequences $\{t_n^j\}_{n\in\N,1\le j<
\bar{J}}\subset\R$, such that the following holds. For any $0\leq
j\leq J<\bar{J}$, let \EQ{
 \pt \cU_n(t):=e^{it\LR D}f_n', \pq \cN_n(t):=e^{it\al D}g_n',
 \pr \lp U_n^j(t):=e^{i(t-t_n^j)\LR D}\lp f^j, \pq \lp N_n^j(t):=e^{i(t-t_n^j)\al D}\lp g^j,
 \pr \pe \cU_n := \cU_n - \sum_{j=1}^J \lp U_n^j, \pq \pe \cN_n := \cN_n - \sum_{j=1}^J \lp N_n^j.}
Then for any $j,k\in\{1 \ldots J\}$, we have
$t_\I^j:=\lim_{n\to\I} t_n^j \in \{0,\pm\I\}$,
\EQ{ \label{separation}
  j\not=k\implies \lim_{n\to\I}|t_n^j-t_n^k|=\I,}
\EQ{ \label{weak conv}
  \pt (\pe \cU_n,\pe \cN_n)(t_n^j)\to 0\ \text{weakly in $H^1\times L^2$} \text{as } n\to \I,
  \pr (\pe \cU_n,\pe \cN_n)(0)\to 0\ \text{weakly in $H^1\times L^2$} \text{as } n\to \I,}
and for $\forall\delta>0$,
\EQ{ \label{smallness}
 \lim_{J\to \bar{J}}\limsup_{n\to\I}
 [\|\pe \cU_n\|_{L^\I_t B^{-1/2-\de}_{\I,2}} + \|\pe \cN_n\|_{L^\I_t(\dot B^{-3/2-\de}_{\I,2}+\dot B^{-3/2+\de}_{\I,2})}] =0.}
\end{lem}

\begin{rem}
1) \eqref{separation}--\eqref{weak conv} implies the linear orthogonality
\EQ{
 \pt \lim_{n\to\I}\left(\|\cU_n(0)\|_{H^1}^2-\sum_{j=1}^J\|\lp U_n^j(0)\|_{H^1}^2-\|\pe \cU_n(0)\|_{H^1}^2\right) = 0,
 \pr \lim_{n\to\I}\left(\|\cN_n(0)\|_{L^2}^2-\sum_{j=1}^J\|\lp N_n^j(0)\|_{L^2}^2-\|\pe \cN_n(0)\|_{L^2}^2\right) = 0,}
as well as the nonlinear orthogonality
\EQ{ \label{linear orth}
 \pt \lim_{n\to\I}\left( K_i(\cU_n(0))-\sum_{j=1}^J K_i(\lp U_n^j(0))-K_i(\pe \cU_n(0))\right)=0, i=0,2,
 \pr \lim_{n\to\I}\left( E(\cU_n(0),\cN_n(0))-\sum_{j=1}^J E(\lp U_n^j(0),\lp N_n^j(0))-E(\pe \cU_n(0),\pe \cN_n(0))\right)=0.}
The same orthogonality holds also along $t=t_n^j$ instead of $t=0$.

2) The norms in \eqref{smallness} are related to the Sobolev embedding $L^2\subset \dot B^{-3/2}_\I$. Interpolation with the Strichartz estimate extends the smallness to any Strichartz norms as far as the exponents are not sharp either in $L^p$ or in regularity (including the low frequency of $\cN$).
\end{rem}

We call such a sequence of free solutions $\{(\lp U_n^j,\lp
N_n^j)\}_{n\in\N}$ a {\it free concentrating wave}. Now we introduce
the nonlinear profile associated to a free concentrating wave \EQ{
 (\lp{U}_n(t),\lp{N}_n(t))=U(t-t_n)(\lp{f},\lp{g}), \pq t_\I=\lim_{n\to\I}t_n\in\{0,\pm\I\},}
where $U(t)=e^{it\LR D}\oplus e^{it\al D}$ denotes the free propagator.
With it, we associate the {\it nonlinear profile} $(\np{U},\np{N})$,
defined as the solution of the Klein-Gordon-Zakharov system satisfying \EQ{
 (\cU,\cN)=U(t)(\lp{f},\lp{g})-i\int_{-t_\I}^t U(t-s)(\LR{D}^{-1}(nu),\al D|u|^2)(s)ds,}
which is obtained  by solving the initial data problem (if $t_\I=0$)
or by solving the final data problem (if $t_\I=\pm\I$).

We call $(\np{U}_n(t),\np{N}_n(t)):=(\np{U}(t-t_n),\np{N}(t-t_n))$
the {\it nonlinear concentrating wave} associated with
$(\lp{U}_n(t),\lp{N}_n(t))$. By the above construction we have
\EQ{
\|(\lp{U}_n,\lp{N}_n)(0)-(\np{U}_n,\np{N}_n)(0)\|_{H^1\times L^2}
 =\|U(-t_n)(\lp{f},\lp{g})-(\np{U},\np{N})(-t_n)\|_{H^1\times L^2}\to 0.}

Given a sequence of solutions to the Klein-Gordon-Zakharov system with bounded
initial data, we can apply the free profile decomposition Lemma
\ref{free prof} to the sequence of initial data, and associate a
nonlinear profile with each free concentrating wave. If all
nonlinear profiles are scattering and the remainder is small enough,
then we can conclude that the original sequence of nonlinear
solutions is also scattering with a global Strichartz bound. More
precisely, we have
\begin{lem}\label{NL profile}
For each free concentrating wave $(\lp U_n^j,\lp N_n^j)$ in Lemma
\ref{free prof}, let $(\np U_n^j,\np N_n^j)$ be the associated
nonlinear concentrating wave. Let $(\cU_n, \cN_n)$ be the
sequence of nonlinear solutions with $(\cU_n,
\cN_n)(0)=(f_n,g_n)$. If $\|(\np
U_n^j,\np N_n^j)\|_{S(0,\I)}<\I$ for all $j< \bar J$, then
\EQ{\limsup_{n\to\I}\|(\cU_n,\cN_n)\|_{S(0,\I)}<\I.}
\end{lem}

To prove Lemma \ref{NL profile}, we need some global stability. In
the next subsection, we will refine the normal form reduction and
the nonlinear estimates that was used in \cite{GNW}, and then prove
Lemma \ref{NL profile} and Lemma \ref{lem:crit}.


\subsection{Nonlinear estimates with small non-sharp norms}
In order to obtain the nonlinear profile decomposition, we need that
the non-sharp smallness \eqref{smallness} is sufficient to reduce
the nonlinear interactions globally. The idea is to use
interpolation, thus we need to do more refined estimates than in
\cite{GNW}, more precisely, to avoid using the sharp (or endpoint) norms with $L^2_t$ or $L^\I_t$.

\subsubsection{Modifying the nonresonant part}
Following the idea of \cite{GNW2}, we modify the resonance decomposition in \cite{GNW} of the bilinear interactions $nu$ and
$|u|^2$ as follows. Let $u=\sum_{k\in\Z} P_ku$ be the standard
homogeneous Littlewood-Paley decomposition such that $\supp\F P_k
u\subset\{2^{k-1}<|\x|<2^{k+1}\}$. For a parameter $\be\ge
k_\al+|\log_2c_\al|+|\log_2\delta_\al|$, where $k_\al$, $c_\al$ and
$\delta_\al$ were given in \cite{GNW} such that the resonance
disappeared for $$\sum_{\substack{|2^k-c_\alpha|>  \delta_\alpha,\\
k\in\Z}}P_k\cN P_{\leq k-k_\alpha}\cU \text{\ \ and\ \ }\sum_{\substack{|2^k-c_\alpha|>  \delta_\alpha,\\
k\in\Z}}P_k\cU P_{\leq k-k_\alpha}\cU.$$
Let \EQ{
 \pt XL:=\{(j,k)\in\Z^2 \mid j\ge \max(k+k_\al,\be) \},
 \pr RL:=\{(j,k)\in\Z^2 \mid |j|<\be \tand k\leq \max(j-k_\al, -\be)\},
 \pr LL:=\{(j,k)\in\Z^2 \mid \max(j,k)\le -\be\},
 \pr LH:=\{(j,k)\in\Z^2 \mid k> \max(j-k_\al, -\be)\},
 \pr HH:=\{(j,k)\in\Z^2 \mid |j-k|<k_\al \tand \max(j,k)\ge\be\},
 \pr RR:=\{(j,k)\in\Z^2 \mid \max(j,k)<\be\},}
and $LX:=\{(k,j)\mid (j,k)\in XL\}$. Then
\EQ{
 \Z^2=(XL \cup LL) \cup (RL \cup LH) = (XL\cup LX)\cup(HH\cup RR),}
where all the unions are disjoint. For any set $A\subset\Z^2$, and any functions $f(x),g(x)$, we denote the bilinear frequency cut-off to $A$ by
\EQ{
 (fg)_A=\F^{-1}\int\cP_A\hat f(\x-\y)\hat g(\y)d\y:=\sum_{(j,k)\in A}(P_jf)(P_kg).}

For the nonlinear term $nu=\cN\cU/4+\bar{\cN}\cU/4+
 \cN\bar{\cU}/4+\bar{\cN}\bar{\cU}/4$, we apply the time integration by parts on $XL\cup LL$, where the phase factors
\begin{align*}
& \om_1=-\LR\x+\al|\x-\y|+\LR\y,
& \om_2=-\LR\x-\al|\x-\y|+\LR\y,\\
& \om_3=-\LR\x+\al|\x-\y|-\LR\y,
& \om_4=-\LR\x-\al|\x-\y|-\LR\y.
\end{align*}
are estimated \EQ{ \om_1,\om_2\sim_\alpha |\x-\y|\text{\ \ and\ \ }
\om_3,\om_4\sim_\alpha \LR\xi\text{ \ in\  }XL\cup LL,}
both of which are gained in the bilinear
operators \EQ{
 \Om_i(f,g):=\F^{-1}\int \cP_{XL\cup LL}\frac{\hat f(\x-\y)\hat g(\y)}{\om_i}d\y,\ i=1,2,3,4.}
For the nonlinear term $u\bar{u}=\cU\bar{\cU}/4+\bar{\cU}\cU/4+\cU^2/4+\bar{\cU}^2/4$, we integrate by parts on $XL\cup LX$. Then we get a bilinear operator of the form
\begin{align}
\tilde\Om_i(f,g)&:=\F^{-1}\int \cP_{XL\cup LX}\frac{\hat f(\x-\y)\hat{\bar
g}(\y)}{\tilde{\om}_i}d\y,
\end{align}
where
\begin{align*}
& \tilde{\om}_1=-\al|\x|+\LR{\x-\y}-\LR\y,
& \tilde{\om}_2=-\al|\x|-\LR{\x-\y}+\LR\y,\\
& \tilde{\om}_3=-\al|\x|+\LR{\x-\y}+\LR\y,
& \tilde{\om}_4=-\al|\x|-\LR{\x-\y}-\LR\y.
\end{align*}
Since $\om_j$ and $\ti\om_j$ are in the dual relation with the correspondence $\x \mapsto \y-\x$, we have \EQ{ \tilde{\om}_1,\tilde{\om}_2\sim_\alpha |\x|\text{\ \ and\ \ }
\tilde{\om}_3,\tilde{\om}_4\sim_\alpha \LR{\x-\y}\text{ \ in\  }XL\cup LL.}

In order to simplify the presentation, we assume that $\alpha<1$\footnote{This is the physical case in plasma} and the nonlinear terms in the first and second equation of \eqref{eq:KGZ2} are $\cN\cU$ and $\cU\bar\cU$ respectively. For other cases, the proof is almost the same.

After this modification of the normal form, we can rewrite the
integral equation for \eqref{eq:KGZ2} as follows. Let \EQ{
 \vec \cU:=(\cU,\cN),
 \pq \vec \cU^0:=U(t)\vec \cU(0)=(e^{it\LR D}\cU(0),e^{it\al D}\cN(0)).}
For the fixed free solution $\vec \cU^0$, the iteration $\vec
\cU'\mapsto \vec \cU$ is given by \EQ{
 \vec \cU=\vec \cU^0+U(t)B(\vec \cU(0),\vec \cU(0))-B(\vec \cU',\vec \cU')-Q(\vec \cU',\vec \cU')-T(\vec \cU',\vec \cU',\vec \cU'),}
where the bilinear forms $B,Q$ and the trilinear form $T$ are
defined by
\begin{align*}
B(\vec \cU_1,\vec \cU_2):=&(\Om(\cN_1,\cU_2),D\ti\Om(\cU_1, \cU_2)),\\
Q(\vec \cU_1,\vec \cU_2):=&i\int_0^t
U(t-s)(\LR D^{-1}(\cN_1\cU_2)_{LH\cup RL},D(\cU_1\ba \cU_2)_{HH\cup RR})(s)ds,\\
T(\vec \cU_1,\vec \cU_2,\vec
\cU_3):=&i\int_0^tU(t-s)(\al \LR D^{-1}\Om(D(\cU_1\ba \cU_2),\cU_3)+ \LR D^{-1}\Om(\cN_1, \LR D^{-1}(\cN_2\cU_3)),\\
&\ \ \ \ \ \ \ \ \ \ \ \ \ \ \ \ \ \ \ \ D\ti\Om(\cU_1, \LR
D^{-1}(\cN_2\cU_3))(s))ds.
\end{align*}
where $\Om=\Om_1$ and $\ti\Om=\ti\Om_1$.
For brevity, we denote
\begin{align*}
NL(\vec \cU_1,\vec \cU_2,\vec \cU_3):=&B(\vec \cU_1,\vec
\cU_2)+Q(\vec \cU_1,\vec \cU_2)+T(\vec \cU_1,\vec \cU_2,\vec
\cU_3), \ \  NL(\vec \cU):=NL(\vec \cU,\vec\cU,\vec\cU), \\
B(\vec \cU):=&B(\vec \cU,\vec\cU),\ \ Q(\vec \cU):=Q(\vec
\cU,\vec\cU),\ \ T(\vec \cU):=T(\vec \cU,\vec\cU,\vec\cU).
\end{align*}

We can estimate each term in the Duhamel formula using some powers
of Strichartz norms with non-sharp exponents. For brevity of
H\"older-type estimates, we denote the space-time norms by \EQ{
 \pt (b,d,s):=L^{1/b}_t \dot B^s_{1/d,2},
 \pr (b,d\pm \e,s)_+:=(b,d+\e,s)+(b,d-\e,s),
 \pr (b,d\pm \e,s)_\cap:=(b,d+\e,s)\cap(b,d-\e,s).}
Using the above notation, we introduce nearly full sets of the radial Strichartz norms for the Klein-Gordon and the wave equations. Fix small numbers
\EQ{
 0<\ka \ll \e \ll 1,}
and let
\EQ{ \label{def S}
 \pt K:=[(0,\fr 12,0|1)\cap(\fr 12,\fr 3{10}-\fr\ka3,\fr 25-\ka|\fr
7{10}+\fr\ka3)],
 \pr W:=(0,\fr 12,0)\cap(\fr 12,\fr 14-\fr\ka3,-\fr 14-\ka),
 \pr S:=K\times W.}
where $\|\cU\|_{(b,d,s_1|s_2)}:=\|P_{<0}\cU\|_{(b,d,s_1)}+\|P_{\geq0}\cU\|_{(b,d,s_2)}$.
Also we denote the smallness in \eqref{smallness} by using \EQ{
 \|\cU\|_X:=\|\cU\|_{L^\I_t(B^{-\fr 12-\de}_{\I,2})}, \pq \|\cN\|_Y:=\|\cN\|_{L^\I_t(\dot{B}^{-\fr 32-\de}_{\I,2}+\dot{B}^{-\fr 32+\de}_{\I,2})}, \pq Z:=X\times Y.}

In order to control the nonlinear terms by interpolation between $S$ and $Z$,
 we will choose $(b,d,s)$ for $\cU$ and $\cN$ respectively to be $H^s$ admissible with $0<s<1$ and $L^2$ admissible for radial functions.
Moreover, we will choose $b<1/2$ and $(b,d)\not=(0,1/2)$.
Besides that, we will use the sum space\footnote{This is because $\cN(0)\in L^2$ while $\cU(0)\in H^1=L^2\cap\dot H^1$.} with small $\e>0$ for $\cN$
 and the intersection for $\cU$, so that we can dispose of very low or high frequencies, and sum over the dyadic decomposition without any difficulty.

\subsubsection{Bare bilinear terms}
First consider the bilinear terms which do not contain the time integration, namely the boundary term in the transform.
\begin{lem}
(a) There exists $\theta>0$ such that for any $\cN$ and $\cU$, we have
\begin{align}
 \|\LR{D}^{-1}\Om(\cN,\cU)\|_{K} \lec& 2^{-\te \be}\|\cU\|_{K}^{1-\te}\|\cN\|_{W}^{1-\te}\|\cU\|_X^\te\|\cN\|_Y^\te .
\end{align}

(b) There exists $\theta>0$ such that for any $\cU$ and $\cU'$, we have
\begin{align}
 \|D\ti\Om(\cU,\cU')\|_{W} \lec& 2^{-\te \be}\|\cU\|_{K}^{1-\te}\|\cU'\|_{K}^{1-\te}\|\cU\|_X^{\te}\|\cU'\|_X^\te.
\end{align}
\end{lem}

\begin{proof}
(a) For $(j,k)\in XL$, we have only high frequencies. By the Coifman-Meyer-type bilinear estimate on dyadic pieces (see \cite[Lemma 3.5]{GN}), we have
\begin{align*}
\|\LR{D}^{-1}\Om(\cN_j,\cU_k)\|_{(0,\frac{1}{2},0)} &\lesssim \|\LR{D}\Om(\LR{D}^{-1}\cN_j,\cU_k)\|_{L^\I L^2}\\
&\lesssim \|\LR{D}^{-1}\cN_j\|_{(0,\fr 15\pm\e,0)_+}\|\cU_k\|_{(0,\fr 3{10}\pm\e,0)_\cap}\\
 &\lesssim 2^{-\be/10}\|\LR D^{-1}\cN_j\|_{(0,\fr 15\pm\e,\fr 1{10})_+}\|\cU_k\|_{(0,\fr 3{10}\pm\e,0)_\cap},
\end{align*}
and by  non-sharp Sobolev embedding
\begin{align*}
&\|\LR D^{-1}\cN_j\|_{(0,\fr 15\pm\e,\fr 1{10})_+}\lec \|\cN_j\|_{(0,\fr 14,-\fr 34\pm3\e)_+}
\lec\|\cN_j\|^{\fr 12}_{(0,\fr12,0)}\|\cN_j\|^{\fr 12}_{(0,0,-\fr 32\pm6\e )_+},\\
&\|\cU_k\|_{(0,\fr 3{10}\pm\e,0)_\cap}\lec \|\cU_k\|_{(0,\fr 25,\fr
3{10}\pm3\e)_\cap}
\lec\|\cU_k\|^{\fr45}_{(0,\fr12,0|1)}\|\cU_k\|^{\fr15}_{(0,0,-\fr
12-5\e)}.
\end{align*}
Similarly,  for $(j,k)\in LL$, we have only low frequencies and then
\begin{align*}
\|\LR{D}^{-1}\Om(\cN_j,\cU_k)\|_{(0,\frac{1}{2},0)} &\lesssim \|D^{-1}\cN_j\|_{(0,\fr 2{15}\pm\e,0)_+}\|\cU_k\|_{(0,\fr {11}{30}\pm\e,0)_\cap}\\
 &\lesssim 2^{-\be/10}\|D^{-1}\cN_j\|_{(0,\fr 2{15}\pm\e,-\fr 1{10})_+}\|\cU_k\|_{(0,\fr {11}{30}\pm\e,0)_\cap}.
\end{align*}
By non-sharp Sobolev embedding,
\begin{align*}
&\|D^{-1}\cN_j\|_{(0,\fr 2{15}\pm\e,-\fr 1{10})_+}\lec \|\cN_j\|_{(0,\fr 5{12},-\fr 14\pm3\e)_+}
\lec\|\cN_j\|^{\fr 56}_{(0,\fr12,0)}\|\cN_j\|^{\fr 16}_{(0,0,-\fr 32\pm18\e )_+},\\
&\|\cU_k\|_{(0,\fr {11}{30}\pm\e,0)_\cap}\lec \|\cU_k\|_{(0,\fr 25,\fr 1{10}\pm3\e)_\cap}
\lec\|\cU_k\|^{\fr45}_{(0,\fr12,0)}\|\cU_k\|^{\fr15}_{(0,0,-\fr 12-\e)}.
\end{align*}
 Thus we obtain, after summation over dyadic decomposition,
\EQ{
 \|\LR{D}^{-1}\Om(\cN,\cU)\|_{(0,\frac{1}{2},0|1)} \lec& 2^{-\te \be}\|\cU\|_{K}^{1-\te}\|\cN\|_{W}^{1-\te}
 \|\cU\|_X^\te\|\cN\|_Y^\te,}
for some small $\te>0$. Similarly, for $(j,k)\in XL$ we have, \EQ{
 &\|\LR{D}^{-1}\Om(\cN_j,\cU_k)\|_{(\fr 12,\fr 3{10}-\fr\ka3,\fr 7{10}+\fr\ka3)}\\
 \lesssim&\|\LR{D}\Om(\LR{D}^{-1}\cN_j,\cU_k)\|_{(\fr 12,\fr 3{10}-\fr\ka3,-\fr 3{10}+\fr\ka3)}\\
 \lesssim& \|D^{-1}\cN_j\|_{(\fr 14,\fr {7}{30}-\fr \ka 3\pm\e,-\fr1{10} -\ka)_+}\|\cU_k\|_{(\fr 14,\fr 1{15}\pm\e,-\fr15+\fr43 \ka)_\cap}
 \\
 \lesssim& 2^{-\be/5}\|\cN_j\|_{(\fr 14,\fr {7}{30}-\fr \ka 3\pm\e,-\fr9{10} -\ka)_+}\|\cU_k\|_{(\fr 14,\fr 1{15}\pm\e,-\fr15+\fr43 \ka)_\cap}}
 and
\begin{align*}
\|\cN_j\|_{(\fr 14,\fr {7}{30}-\fr \ka 3\pm\e,-\fr9{10} -\ka)_+}\lec& \|\cN_j\|_{(\fr 14,\fr {11}{40}-\fr \ka 6,-\fr{31}{40} -\fr\ka2\pm3\e)_+}
\lec \|\cN_j\|_{(\fr 14,\fr {11}{40}-\fr \ka 6,-\fr{17}{40} -\fr\ka2\pm3\e)_+}\\
\lec&(\|\cN_j\|^{\fr 38}_{(0,\fr12,0)}\|\cN_j\|^{\fr 58}_{(\fr 12,\fr 14-\fr\ka3,-\fr 14-\ka)})^{\fr45}\|\cN_j\|^{\fr 15}_{(0,0,-\fr 32\pm15\e )_+},
\end{align*}
\begin{align*}
\|\cU_k\|_{(\fr 14,\fr 1{15}\pm\e,-\fr15+\fr43 \ka)_\cap}\lec &\|\cU_k\|_{(\fr 14,\fr 14-\fr 1{6}\ka,\fr7{20}+\fr56 \ka\pm3\e)_\cap}
\lec\|\cU_k\|_{(\fr 14,\fr 14-\fr 1{6}\ka,\fr1{20}-\fr \ka2-3\e|\fr25+\fr \ka6-3\e)}\\
\lec&(\|\cU_k\|^{\fr57}_{(\fr 12,\fr 3{10}-\fr\ka3,\fr 25-\ka|\fr 7{10}+\fr\ka3)}
\|\cU_k\|^{\fr27}_{(0,\fr 12,0|1)})^{\fr7{10}}\|\cU_k\|^{\fr3{10}}_{(0,0,-\fr 12-10\e)};
\end{align*}
and for $(j,k)\in LL$ we have
\EQ{
 \pt\|\LR{D}^{-1}\Om(\cN_j,\cU_k)\|_{(\fr 12,\fr 3{10}-\fr\ka3,0)}
 \pn\lec \|D^{-1}\cN_j\|_{(\fr 14,\fr 1{15}-\fr \ka 3\pm\e,0)_+}\|\cU_k\|_{(\fr 14,\fr {7}{30}\pm\e,0)_\cap}
 \pr\pq\lec 2^{-\be/20}\|D^{-1}\cN_j\|_{(\fr 14,\fr 1{15}-\fr \ka 3\pm\e,-\frac{1}{20})_+}\|\cU_k\|_{(\fr 14,\fr {7}{30}\pm\e,0)_\cap} }
 and
 \begin{align*}
\|D^{-1}\cN_j\|_{(\fr 14,\fr 1{15}-\fr \ka 3\pm\e,-\frac{1}{20})_+}\lec& \|\cN_j\|_{(\fr 14,\fr {5}{24}-\fr \ka 6,-\fr{5}{8} -\fr\ka2\pm3\e)_+}\\
\lec&(\|\cN_j\|^{\fr 14}_{(0,\fr12,0)}\|\cN_j\|^{\fr 34}_{(\fr 12,\fr 14-\fr\ka3,-\fr 14-\ka)})^{\fr23}\|\cN_j\|^{\fr 13}_{(0,0,-\fr 32\pm9\e )_+},
\end{align*}
\begin{align*}
\|\cU_k\|_{(\fr 14,\fr {7}{30}\pm\e,0)_\cap}\lec& \|\cU_k\|_{(\fr 14,\fr {7}{30}\pm\e,-\frac{1}{20})_\cap}
\lec\|\cU_k\|_{(\fr 14,\fr {11}{40}-\fr \ka 6,\fr{3}{40} -\fr\ka2-3\e)}\\
\lec&(\|\cU_k\|^{\fr23}_{(\fr 12,\fr 3{10}-\fr\ka3,\fr 25-\ka)}
\|\cU_k\|^{\fr13}_{(0,\fr 12,0)})^{\fr34}\|\cU_k\|^{\fr14}_{(0,0,-\fr 12-12\e)}.
\end{align*}
Hence in either case we can control by non-sharp norms, so
\EQ{
 \|\LR{D}^{-1}\Om(\cN,\cU)\|_{K} \lec& 2^{-\te \be}\|\cU\|_{K}^{1-\te}\|\cN\|_{W}^{1-\te}\|\cU\|_X^\te\|\cN\|_Y^\te .}

(b) We may assume $(j,k)\in XL$, since the other case $LX$ is treated in the same way. Similarly to the above, we have
\EQ{
\|D\ti\Om(\cU_j,\cU'_k)\|_{L^\I L^2}
&\lec \|\cU_j\|_{(0,\fr 14,0)}\|\cU'_k\|_{(0,\fr 14,0)}\\
&\lec 2^{-\be/10} \|\cU_j\|_{(0,\fr 14,\frac{1}{10})}\|\cU'_k\|_{(0,\fr 14,0)}.}
Note that we have only high frequencies for $\cU$, we have
\begin{align*}
&\|\cU_j\|_{(0,\fr 14,\frac{1}{10})}
\lec
\|\cU_j\|^{\fr12}_{(0,\fr 12,1)}\|\cU_j\|^{\fr12}_{(0,0,-\fr 12-\e)},\\
&\|\cU'_k\|_{(0,\fr 14,0)}
\lec
\|\cU'_k\|^{\fr12}_{(0,\fr 12,0|1)}\|\cU'_k\|^{\fr12}_{(0,0,-\fr 12-\e)}.
\end{align*}
Hence
\EQ{
  \|D\ti\Om(\cU,\cU')\|_{L^\I L^2} \lec& 2^{-\te \be}\|\cU\|_{K}^{1-\te}\|\cU'\|_{K}^{1-\te}\|\cU\|_X^{\te}\|\cU'\|_X^\te.}
Similarly, we have
\EQ{
\|D\ti\Om(\cU_j,\cU'_k)\|_{(\fr 12,\fr 14-\fr\ka3,-\fr 14-\ka)}
&\lec\|\cU_j\|_{(\fr 14,\fr 18-\fr \ka3,-\frac18-\ka)}\|\cU'_k\|_{(\fr 14,\fr 18,-\fr18)}
\\&\lec 2^{-\be/10}\|\cU_j\|_{(\fr 14,\fr 18-\fr \ka3,-\frac1{40}-\ka)}\|\cU'_k\|_{(\fr 14,\fr 18,-\fr18)}}
and
\begin{align*}
&\|\cU_j\|_{(\fr 14,\fr 18-\fr \ka3,-\frac1{40}-\ka)}\lec \|\cU_j\|_{(\fr 14,\fr 3{20}-\fr \ka{6},\fr1{10}-\fr \ka2)}
\lec\|\cU_j\|^{\fr12}_{(\fr 12,\fr 3{10}-\fr\ka3,\fr 7{10}+\fr\ka3)}
\|\cU_j\|^{\fr12}_{(0,0,-\fr 12-2\ka)};\\
&\|\cU'_k\|_{(\fr 14,\fr 18,-\fr18)}\lec \|\cU'_k\|_{(\fr 14,\fr 3{20}-\fr \ka{6},-\fr1{20}-\fr \ka2)}
\lec\|\cU'_k\|^{\fr12}_{(\fr 12,\fr 3{10}-\fr\ka3,\fr 25-\ka|\fr 7{10}+\fr\ka3)}
\|\cU'_k\|^{\fr12}_{(0,0,-\fr 12-2\ka)};
\end{align*}
so \EQ{\|D\ti\Om(\cU,\cU')\|_{W} \lec& 2^{-\te \be}\|\cU\|_{K}^{1-\te}\|\cU'\|_{K}^{1-\te}\|\cU\|_X^{\te}\|\cU'\|_X^\te.}
Thus the proof is completed.
\end{proof}

\subsubsection{Duhamel bilinear terms}
Next we consider the remaining bilinear terms in the Duhamel form after the normal form transform.
Here we have to use the radial improvement of the Strichartz norms.
For brevity, we denote the integrals in the Duhamel formula by
\EQ{
 I_{\cU} f:=\int_0^t e^{i(t-s)\LR D}f(s)ds, \pq I_{\cN}f:=\int_0^t e^{i(t-s)\al D}f(s)ds.}

\begin{lem}
(a) There exists $\theta>0$ and $C(\be)>1$ such that for any $\cN$ and $\cU$, we have
\begin{align*}
\|I_{\cU}\LR D^{-1}(\cN\cU)_{LH}\|_{K}
 \le &C(\be) \|\cU\|_{K}^{1-\te}\|\cN\|_{W}^{1-\te} \|\cU\|_X^{\te}\|\cN\|_Y^{\te}
,\\
\|I_{\cU}\LR D^{-1}(\cN\cU)_{RL}\|_{K}
 \le& C(\be) \|\cU\|_{K}^{1-\te}\|\cN\|_{W}^{1-\te}\|\cU\|_X^{\te}\|\cN\|_Y^{\te}.
\end{align*}

(b) There exists $\theta>0$ and $C(\be)>1$ such that for any $\cU$ and $\cU'$, we have
\begin{align*}
\|I_{\cN}D(\cU\cU')_{HH}\|_{W} \le& C(\be)
\|\cU\|_{K}^{1-\te}\|\cU'\|_{K}^{1-\te}\|\cU\|_X^{\te}\|\cU'\|_X^{\te},\\
\|I_{\cN}(\cU\cU')_{RR}\|_{W} \le& C(\be)
\|\cU\|_{K}^{1-\te}\|\cU'\|_{K}^{1-\te}\|\cU\|_X^{\te}\|\cU'\|_X^{\te}.
\end{align*}
\end{lem}

\begin{proof}
In this proof we ignore the dependence of the constants on $\be$.

(a) For $(j,k)\in LH$, we have for $0\le s\le 1+2\e$,
\EQ{\label{dual}
 &\|\LR D^{-1}(\cN_j\cU_k)\|_{(1-2\e,\fr 12+2\e,s+2\e)}\\
 \lec &\|\cN_j\|_{(\fr 12-\e,\fr 14-\e\pm\fr {\e^2} 3,-\fr 14-4\e)_+}
  \|\LR D^{-1}\cU_k\|_{(\fr 12-\e,\fr 14+3\e\pm\fr {\e^2} 3,s+\fr 14+6\e)_\cap}
 \\
 \lec &\|\cN_j\|_{(\fr 12-\e,\fr 14-\e\pm\fr {\e^2} 3,-\fr 14-2\e)_+}
  \|\LR D^{-1}\cU_k\|_{(\fr 12-\e,\fr 14+3\e\pm\fr {\e^2} 3,\fr 54+8\e)_\cap},}
where in the second inequality we used that $k$ is bounded from below.
Since
\begin{align*}
\|\cN_j\|_{(\fr 12-\e,\fr 14-\e\pm\fr {\e^2} 3,-\fr 14-4\e)_+}\lec &\|\cN_j\|_{(\fr 12-\e,\fr 14-\fr\e2-\fr13 (1-2\e)\ka,-\fr 14-\fr 52\e-(1-2\e)\ka\pm\e^2)_+}\\
\lec&\|\cN_j\|^{1-2\e}_{(\fr 12,\fr 14-\fr\ka3,-\fr 14-\ka)}\|\cN_j\|^{2\e}_{(0,0,-\fr 32\pm \fr\e2 )_+},
\end{align*}
\begin{align*}
 \|\LR D^{-1}\cU_k\|_{(\fr 12-\e,\fr 14+3\e\pm\fr {\e^2} 3,\fr 54+8\e)_\cap}
\lec&  \|\cU_k\|_{(\fr 12-\e,\fr 14+3\e\pm\fr {\e^2} 3,\fr 14+8\e)_\cap}\\
\lec&  \|\cU_k\|_{(\fr 12-\e,\fr 3{10}-\fr35\e-\fr13 (1-2\e)\ka,\fr 25-\fr{14}5\e-(1-2\e)\ka\pm\e^2)_\cap}\\
\lec&  \|\cU_k\|_{(\fr 12-\e,\fr 3{10}-\fr35\e-\fr13 (1-2\e)\ka,\fr 7{10}-\fr{12}5\e+\fr13(1-2\e)\ka-\e^2)}\\
\lec&\|\cU_k\|^{1-2\e}_{(\fr 12,\fr 3{10}-\fr\ka3,\fr 7{10}+\fr\ka3)}
\|\cU_k\|^{2\e}_{(0,0,-\fr 12-\fr\e2)};
\end{align*}
and the left hand side of \eqref{dual} is $\dot H^s$-admissible norm for the Strichartz estimate (without the radial symmetry), we obtain the full Strichartz
bound in $H^1$.

For $(j,k)\in RL$, we only have the low frequencies of $\cU$ and we may neglect the regularity of $\cN$ and the product. Using the radial improved Strichartz,
\EQ{
 \|\LR D^{-1}(\cN_j\cU_k)\|_{(\fr 12+2\e,\fr 34-3\e,0)}
 \lec \|\cN_j\|_{(\fr 12-\e,\fr 14-\fr 23 \e,0)}
  \|\cU_k\|_{(3\e,\fr 12-\fr 73\e,-\e^2)}}
and
\begin{align*}
 \|\cU_k\|_{(3\e,\fr 12-\fr 73\e,-\e^2)}
\lec&  \|\cU_k\|_{(3\e,\fr 12-\fr {17}{10}\e-2\ka\e,\fr {19}{10}\e-6\ka\e-\e^2)}\\
\lec&(\|\cU_k\|^{\fr{6\e}{1-\e}}_{(\fr 12,\fr 3{10}-\fr\ka3,\fr 25-\ka)}
\|\cU_k\|^{\fr{1-7\e}{1-\e}}_{(0,\fr 12,0)})^{1-\e}\|\cU_k\|^{\e}_{(0,0,-\fr 12-\e)};
\end{align*}
Summing these estimates over dyadic pieces in the specified regions, and using non-sharp Sobolev embedding and interpolation, we obtain
\begin{align*}
\|I_{\cU}\LR D^{-1}(\cN\cU)_{LH}\|_{K}
 \le &C(\be) \|\cU\|_{K}^{1-\te}\|\cN\|_{W}^{1-\te} \|\cU\|_X^{\te}\|\cN\|_Y^{\te}
,\\
\|I_{\cU}\LR D^{-1}(\cN\cU)_{RL}\|_{K}
 \le& C(\be) \|\cU\|_{K}^{1-\te}\|\cN\|_{W}^{1-\te}\|\cU\|_X^{\te}\|\cN\|_Y^{\te}.
\end{align*}

(b) We consider only the case $j\ge k$ for $\cU_j \cU_k'$, since the other case is treated in the same way. For $(j,k)\in HH$, there are only high frequencies
for both $\cU$ and  $\cU'$. Hence, we have
\EQ{
 \|\cU_j \cU'_k\|_{(1-2\e,\fr 12+2\e,1+4\e)}
 \lec \|\cU_j\|_{(\fr 12- \e ,\fr 14+ \e ,\fr 12+2\e)} \|\cU'_k\|_{(\fr 12-\e,\fr 14+\e,\fr 12+2\e)}}
 and
 \begin{align*}
  \|\cU_j\|_{(\fr 12- \e ,\fr 14+ \e ,\fr 12+2\e)}
\lec&  \|\cU_j\|_{(\fr 12-\e,\fr 3{10}-\fr35\e-\fr13 (1-2\e)\ka,\fr {13}{20}-\fr{14}5\e-(1-2\e)\ka)}\\
\lec&  \|\cU_j\|_{(\fr 12-\e,\fr 3{10}-\fr35\e-\fr13 (1-2\e)\ka,\fr 7{10}-\fr{12}5\e+\fr13(1-2\e)\ka-\e^2)}\\
\lec&\|\cU_j\|^{1-2\e}_{(\fr 12,\fr 3{10}-\fr\ka3,\fr 7{10}+\fr\ka3)}
\|\cU_j\|^{2\e}_{(0,0,-\fr 12-\fr\e2)};
\end{align*}

In the case $(j,k)\in RR$, since $j$ is bounded from above,
\EQ{
  \|\cU_j \cU'_k\|_{(\fr 12+\e,\fr 34,\fr 54+\e)}
 \lec \|\cU_j\|_{(\fr 12-\e,\fr 14+\fr73\e,\fr 54+\e+\e^2)} \|\cU'_k\|_{(3\e,\fr 12-\fr73\e,-\e^2)}.}
 There are only low frequencies in this case.
Hence
\begin{align*}
 \|\cU_j\|_{(\fr 12-\e,\fr 14+\fr73\e,\fr 54+\e+\e^2)}
\lec&  \|\cU_j\|_{(\fr 12-\e,\fr {3}{10}(1-2\e)-\fr13 (1-2\e)\ka,\fr {2}{5}-\fr {9}{5}\e-(1-2\e)\ka-\e^2)}\\
\lec&\|\cU_j\|^{1-2\e}_{(\fr 12,\fr 3{10}-\fr\ka3,\fr 25-\ka)}
\|\cU_j\|^{2\e}_{(0,0,-\fr 12-\fr12\e)};
\end{align*}
and then
\begin{align*}
\|I_{\cN}D(\cU\cU')_{HH}\|_{W} \le& C(\be)
\|\cU\|_{K}^{1-\te}\|\cU'\|_{K}^{1-\te}\|\cU\|_X^{\te}\|\cU'\|_X^{\te},\\
\|I_{\cN}D(\cU\cU')_{RR}\|_{W} \le& C(\be)
\|\cU\|_{K}^{1-\te}\|\cU'\|_{K}^{1-\te}\|\cU\|_X^{\te}\|\cU'\|_X^{\te}.
\end{align*}
\end{proof}

\subsubsection{Duhamel trilinear terms}
Finally we estimate the trilinear terms which appear after the
normal transform. These are supposedly the easiest, but there is a
small complication due to the fact that we have to use negative
Sobolev spaces for $\cN$ in some of the products: \EQ{
 \pt \|fg\|_{\dot B^{-s}_{r,2}} \lec \|f\|_{\dot B^{-s}_{p,2}}\|g\|_{\dot B^s_{q,2}}
 \pr 0\le s<3/q,\ 1/r=1/p+1/q-s/3.}
In the next lemma, the constant may decay as $\be\to\I$, but we do not need it.
\begin{lem}
(a) There exists $\theta>0$ such that for any $\cU,\cV,\cW,\cN,\cN'$, we have
\begin{align*}
 \|I_{\cU}\LR D^{-1}\Om(D(\cU\cV),\cW)\|_{K} \lec&
 \|\cU\|_{K}^{1-\te}\|\cV\|_{K}^{1-\te}\|\cW\|_{K}^{1-\te}\|\cU\|_X^{\te}\|\cV\|_X^{\te}\|\cW\|_X^{\te},\\
 \|I_{\cU}\LR D^{-1}\Om(\cN,\LR D^{-1}(\cN'\cU))\|_{K}
 \lec& \|\cN\|_{W}^{1-\te}\|\cN'\|_{W}^{1-\te}\|\cU\|_{K}^{1-\te}\|\cN\|_Y^{\te}\|\cN'\|_Y^{\te}\|\cU\|_X^{\te}.
\end{align*}

(b) There exists $\theta>0$ such that for any $\cN,\cU,\cU'$, we have
\begin{align*}
&\|I_{\cN}D\ti\Om(\LR D^{-1}(\cN\cU),\cU')\|_{W}+\|I_{\cN}D\ti\Om(\cU,\LR D^{-1}(\cN\cU'))\|_{W}\\
 \lec&
\|\cN\|_{W}^{1-\te}\|\cU\|_{K}^{1-\te}\|\cU'\|_{K}^{1-\te}\|\cN\|_Y^{\te}\|\cU\|_X^{\te}\|\cU'\|_X^{\te}.
\end{align*}
\end{lem}

\begin{proof}
(a) Since
\EQ{
 &\|\LR D^{-1}\Om(D(\cU\cV)_j,\cW_k)\|_{(1,\frac{1}{2},1)} =\|\Om(D(\cU\cV)_j,\cW_k)\|_{(1,\frac{1}{2},0)}\\
  &\lesssim\|(\cU\cV)_j\cW_k)\|_{(1,\frac{1}{2},0)}
  \lesssim \|\cU\|_{(\frac{1}{3},\frac{1}{6},0)}\|\cV\|_{(\frac{1}{3},\frac{1}{6},0)}\|\cW\|_{(\frac{1}{3},\frac{1}{6},0)},}
by  non-sharp Sobolev embedding and interpolation,
\begin{align*}
\|\cU\|_{(\frac{1}{3},\frac{1}{6},0)}\lec &\|\cU\|_{(\fr 13,\fr 4{15}-\fr 2{9}\ka,\fr3{10} -\fr23\ka)}\lec\|\cU\|_{(\fr 13,\fr 4{15}-\fr 2{9}\ka,\fr1{6}-\fr 23\ka-\e|\fr12+\fr 29\ka-\e)}\\
\lec&(\|\cU\|^{\fr56}_{(\fr 12,\fr 3{10}-\fr\ka3,\fr 25-\ka|\fr 7{10}+\fr\ka3)}
\|\cU\|^{\fr16}_{(0,\fr 12,0|1)})^{\fr45}\|\cU\|^{\fr15}_{(0,0,-\fr 12-5\e)}.
\end{align*}
Hence
\EQ{
  \|I_{\cU}\LR D^{-1}\Om(D(\cU\cV),\cW)\|_{K} \lec&
 \|\cU\|_{K}^{1-\te}\|\cV\|_{K}^{1-\te}\|\cW\|_{K}^{1-\te}\|\cU\|_X^{\te}\|\cV\|_X^{\te}\|\cW\|_X^{\te}.}
For $\LR D^{-1}\Om(\cN_j,\LR D^{-1}(\cN'\cU)_k)$, if $(j,k)\in XL$, we have
\EQ{
&\|\LR D^{-1}\Om(\cN_j,\LR D^{-1}(\cN'\cU)_k)\|_{(1,\frac{1}{2},1)}=\|\Om(\cN_j,\LR D^{-1}(\cN'\cU)_k)\|_{(1,\frac{1}{2},0)}\\
 \lec& \|D^{-1}\cN_j\|_{(\fr 12-\e,2\e\pm\fr {\e^2}6,5\e)_+} \|\LR D^{-1}(\cN'\cU)_k\|_{(\fr 12+\e,\fr 12-2\e\pm\fr {\e^2}6,-5\e)_\cap}\\
 \lec& \|D^{-1}\cN_j\|_{(\fr 12-\e,2\e\pm\fr {\e^2}6,5\e)_+} \|(\cN'\cU)_k\|_{(\fr 12+\e,\fr 12-2\e\pm\fr {\e^2}6,-\fr 38-5\e)_\cap}\\
 \lec& \|D^{-1}\cN_j\|_{(\fr 12-\e,2\e\pm\fr {\e^2}6,5\e)_+} \|\cN\|_{(2\e,\fr 38-\fr73\e\pm\fr {\e^2}6,-\fr 38-5\e)_+}
 \|\cU\|_{(\fr12-\e,\fr 14+2\e\pm\fr {\e^2}3,\fr 38+5\e)_\cap},}
where we used the product estimate for negative Sobolev spaces for $\cN'\cU$;
and
 \begin{align*}
 &\|\cU\|_{(\fr12-\e,\fr 14+2\e\pm\fr {\e^2}3,\fr 38+5\e)_\cap}
\lec  \|\cU\|_{(\fr 12-\e,\fr 3{10}-\fr35\e-\fr13 (1-2\e)\ka,\fr {21}{40}-\fr{14}5\e-(1-2\e)\ka\pm\e^2)_\cap}\\
\lec&  \|\cU\|_{(\fr 12-\e,\fr 3{10}-\fr35\e-\fr13 (1-2\e)\ka,\fr
{2}{5}-\fr{9}5\e-(1-2\e)\ka-\e^2|
\fr {7}{10}-\fr{12}5\e+\fr13(1-2\e)\ka-\e^2)}\\
\lec&\|\cU\|^{1-2\e}_{(\fr 12,\fr 3{10}-\fr\ka3,\fr 25-\ka|\fr 7{10}+\fr\ka3)}
\|\cU\|^{2\e}_{(0,0,-\fr 12-\fr\e2)},
\end{align*}
\begin{align*}
\|D^{-1}\cN_j\|_{(\fr 12-\e,2\e\pm\fr {\e^2}6,5\e)_+}\lec &\|\cN_j\|_{(\fr 12-\e,\fr 14-\fr\e2-\fr13 (1-2\e)\ka,-\fr 14-\fr 52\e-(1-2\e)\ka\pm\fr12\e^2)_+}\\
\lec&\|\cN_j\|^{1-2\e}_{(\fr 12,\fr 14-\fr\ka3,-\fr 14-\ka)}\|\cN_j\|^{2\e}_{(0,0,-\fr 32\pm \fr\e4 )_+},
\end{align*}
\begin{align*}
\|\cN\|_{(2\e,\fr 38-\fr73\e\pm\fr {\e^2}6,-\fr 38-5\e)_+}\lec &\|\cN\|_{(2\e,\fr 12-\fr32\e-\fr43\ka\e,-\fr 52\e-4\ka\e\pm\fr12\e^2)_+}\\
\lec&(\|\cN\|^{\fr{4\e}{1-\e}}_{(\fr 12,\fr 14-\fr\ka3,-\fr
14-\ka)}\|\cN\|^{\fr{1-5\e}{1-\e}}_{(0,\fr
12,0)})^{1-\e}\|\cN\|^{\e}_{(0,0,-\fr 32\pm \fr\e2 )_+}.
\end{align*}
If $(j,k)\in LL$, we have
\EQ{
&\|\Om(\cN_j,\LR D^{-1}(\cN'\cU)_k)\|_{L^1L^2}\\
 \lec& \|D^{-1}\cN_j\|_{(\fr 12-\e,\fr\e3\pm\fr {\e^2}6,0)_+} \|(\cN'\cU)_k\|_{(\fr 12+\e,\fr 12-\fr\e3\pm\fr {\e^2}6,0)_\cap}\\
 \lec& \|D^{-1}\cN_j\|_{(\fr 12-\e,2\e\pm\fr {\e^2}6,5\e)_+} \|(\cN'\cU)_k\|_{(\fr 12+\e,\fr 12-\fr\e3\pm\fr {\e^2}6,-\fr 38-5\e)_\cap}\\
 \lec& \|D^{-1}\cN_j\|_{(\fr 12-\e,2\e\pm\fr {\e^2}6,5\e)_+} \|\cN\|_{(2\e,\fr 38-\fr73\e\pm\fr {\e^2}6,-\fr 38-5\e)_+}
 \|\cU\|_{(\fr12-\e,\fr 14+\fr{11}3\e\pm\fr {\e^2}3,\fr 38+5\e)_\cap}.}
 By  non-sharp Sobolev embedding,
\begin{align*}
 &\|\cU\|_{(\fr12-\e,\fr 14+\fr{11}3\e\pm\fr {\e^2}3,\fr 38+5\e)_\cap}
\lec  \|\cU\|_{(\fr 12-\e,\fr 3{10}-\fr35\e-\fr13 (1-2\e)\ka,\fr {21}{40}-\fr{39}5\e-(1-2\e)\ka\pm\e^2)_\cap}\\
\lec&  \|\cU\|_{(\fr 12-\e,\fr 3{10}-\fr35\e-\fr13 (1-2\e)\ka,\fr
{2}{5}-\fr{9}5\e-(1-2\e)\ka-\e^2|\fr
{7}{10}-\fr{12}5\e+\fr13(1-2\e)\ka-\e^2)}.
\end{align*}
Hence
\EQ{
 \|I_{\cU}\LR D^{-1}\Om(\cN,\LR D^{-1}(\cN'\cU))\|_{K}
 \lec& \|\cN\|_{W}^{1-\te}\|\cN'\|_{W}^{1-\te}\|\cU\|_{K}^{1-\te}\|\cN\|_Y^{\te}\|\cN'\|_Y^{\te}\|\cU\|_X^{\te}.}

(b)For $D\ti\Om(\LR D^{-1}(\cN\cU),\cU')$, we have
\EQ{
 &\|D\ti\Om(\LR D^{-1}(\cN\cU),\cU')\|_{L^1L^2} \\
 \lec &\|\LR D^{-1}(\cN\cU)\|_{(\fr 12+\e,\fr 38-2\e,0)}\|\cU'\|_{(\fr 12-\e,\fr 18+2\e,0)}\\
 \lec &\|\cN\cU\|_{(\fr 12+\e,\fr 12-\fr \e3,-\fr 38-5\e)_+}
 \|\cU'\|_{(\fr 12-\e,\fr 14+\fr {11}3\e,\fr 38+5\e)_\cap}\\
  \lec&\|\cN\|_{(2\e,\fr 38-\fr73\e\pm\fr {\e^2}6,-\fr 38-5\e)_+}
 \|\cU\|_{(\fr12-\e,\fr 14+\fr{11}3\e\pm\fr {\e^2}3,\fr 38+5\e)_\cap}\|\cU'\|_{(\fr12-\e,\fr 14+\fr{11}3\e,\fr 38+5\e)},}
where we used the product estimate twice, but did not use any restriction on $j,k$. Hence we have the same estimate on $D\ti\Om(\cU,\LR D^{-1}(\cN\cU'))$, and so
\begin{align*}
&\|I_{\cN}D\ti\Om(\LR D^{-1}(\cN\cU),\cU')\|_{W}+\|I_{\cN}D\ti\Om(\cU,\LR D^{-1}(\cN\cU'))\|_{W}\\
 \lec&
\|\cN\|_{W}^{1-\te}\|\cU\|_{K}^{1-\te}\|\cU'\|_{K}^{1-\te}\|\cN\|_Y^{\te}\|\cU\|_X^{\te}\|\cU'\|_X^{\te}.
\end{align*}
Thus, the proof is completed.
\end{proof}

Note that in the above estimates we needed the $L^\I_t$-type norms only for the bare bilinear terms, but not for the Duhamel terms. Thus we have obtained
\begin{lem}\label{B-Q-T}
There exist $\te>0$, $\y>0$ and $C(\be)>1$ such that for each $\be\gg 1$ and any $\vec \cU_1,\vec \cU_2,\vec \cU_3$, we have
\EQ{\label{NL est}
 \pt2^{\te\be}\|B(\vec \cU_1,\vec \cU_2)\|_S+\|Q(\vec \cU_1,\vec \cU_2)\|_S/C(\be) \lec \|\vec \cU_1\|_S^{1-\te}\|\vec \cU_2\|_S^{1-\te}\|
 \vec \cU_1\|_Z^\te\|\vec \cU_2\|_Z^\te,
 \pr\|T(\vec \cU_1,\vec \cU_2,\vec \cU_3)\|_S \lec \|\vec \cU_1\|_S^{1-\te}\|\vec \cU_2\|_S^{1-\te}\|\vec \cU_3\|_S^{1-\te}\|\vec \cU_1\|_Z^\te
 \| \vec \cU_2\|_Z^\te\|\vec \cU_3\|_Z^\te.}
For the Duhamel terms we have also
\EQ{\label{Q-T}
 \pt\|Q(\vec \cU_1,\vec \cU_2)\|_S \lec C(\be)\|\vec \cU_1\|_{\widetilde{S}}\|\vec \cU_2\|_{\widetilde{S}},
 \pr\|T(\vec \cU_1,\vec \cU_2,\vec \cU_3)\|_S \lec \|\vec \cU_1\|_{\widetilde{S}}\|\vec \cU_2\|_{\widetilde{S}}\|\vec \cU_3\|_{\widetilde{S}},}
 where
 \EQ{
 \pt \widetilde{S}:=\widetilde{K}\times \widetilde{W},
 \pr \widetilde{K}:=[(\eta,\fr 12-\fr 25\eta,
 \fr 45\eta|1- \fr
35\eta+\fr83\ka\eta)\cap(\fr 12,\fr 3{10}-\fr\ka3,\fr 25-\ka|\fr
7{10}+\fr\ka3)],
 \pr \widetilde{W}:=(\eta,\fr 12-\fr 12\eta,
 -\fr 14\eta)\cap(\fr 12,\fr 14-\fr\ka3,-\fr 14-\ka).
 }
\end{lem}

\subsection{Nonlinear profile approximation} We will prove Lemma \ref{NL
profile} and then prove Lemma \ref{lem:crit} by the following two lemmas.

\begin{lem}[Stability]\label{Stability}
For any $A>0$ and $\sigma>0$, there exists $\varsigma>0$ with the following property:
Suppose that $ \vec  \cU_a$ satisfies
$\|\vec \cU_a\|_{S(0,\I)}\leq A$ and approximately solves the Klein-Gordon-Zakharov system in the sense that
\begin{equation*}
\begin{split}
\vec \cU_a=U(t)\vec \cU_a(0)+U(t)B(\vec \cU_a(0))-NL(\vec \cU_a)+\vec e,\quad \|\vec e\|_{S(0,\I)}\leq \varsigma.
\end{split}
\end{equation*}
Then for any initail data $\vec \cU(0)$ satisfying $\|\vec \cU(0)-\vec \cU_a(0)\|_{H^1\times L^2}<\varsigma$, there is a unique global solution $\vec \cU$ satisfying
$\|\vec \cU-\vec \cU_a\|_{S(0,\I)}< \sigma$.
\end{lem}

With $J$ close to $\bar{J}$ and large $n$, our approximate solution is given by
\EQ{
 \vec{\cU}_n^J=(\cU_n^J,\cN_n^J):= \sum_{j=1}^J(\np U_n^j,\np N_n^j)+(\pe \cU_n,\pe \cN_n).}
To prove Lemma \ref{NL profile}, we only need to  prove that $\vec \cU_n^J$ is an approximate solution of the Klein-Gordon-Zakharov system. In fact, we have
\begin{lem}\label{per}
Suppose that $\|(\np U_n^j,\np N_n^j)\|_S<\I$ for all $j<\bar J$, then
\begin{align*}
\lim_{J\to \bar{J}}\limsup_{n\to\I} \|U(t)B(\vec \cU_n^J(0))-NL(\vec \cU_n^J) -\sum_{j=1}^J[U(t)B(\vec{\np U}_n^j(0))-NL(\vec{\np U}_n^j)]\|_{S}=0.
\end{align*}
\end{lem}

The proofs of Lemma \ref{Stability} and Lemma \ref{per} are almost the same with Zakharov case in \cite{GNW2} since we have  got the same  nonlinear estimates with the latter. By these two lemmas, we can prove Lemma \ref{NL
profile} and Lemma \ref{lem:crit} similarly as \cite{GNW2}; and hence we omit the details here.

\section{Rigidity Theorem}

The main purpose of this section is to disprove the existence of
critical element that was constructed in the previous section under the assumption $E^*<J(Q)$. The
main tool is the spatial localization of the virial identity. We
 prove

\begin{thm}[Rigidity Theorem]\label{rigidity}
Let $(\cU,\cN)$ be a global solution to \eqref{eq:KGZ2} satisfying
$K_i(\cU)\geq 0$ for $i=0,2$, and $E(\cU,\cN)< J(Q)$. Moreover, assume $\{(\cU,\cN)(t):t\in \R\}$ is
precompact in $H^1\times L^2$. Then $\cU=\cN\equiv 0$.
\end{thm}

\begin{proof}
By contradiction, we assume $(\cU,\cN)\ne (0,0)$. Then by the
compactness we may assume further $\cU\ne 0$, since otherwise $\cN$ would be a free wave and dispersive. We divide the proof into
the following three steps:

{\bf Step 1:} Energy trapping.

We claim that for any $\varepsilon>0$ there exists $C>1$ such that
\EQ{\label{2norm bdd}
 \|u(t)\|_{L^2}^2\leq C \|\nabla u(t)\|_{L^2}^2+\varepsilon ( \|\dot{u}(t)\|_{L^2}^2+ \frac{1}{2}\|n(t)\|_{L^2}^2+ \frac{1}{2\alpha^2}\|\dot{n}(t)\|_{\dot{H}^{-1}}^2),}
 for all $t\in \mathbb{R}$. Furthermore, there exist $c>0$ such that
\EQ{\label{average+}
  \int_0^t (K_2(u(t))+\|\dot{n}(t)\|_{\dot{H}^{-1}}^2+\|n-u^2\|_{L^2}^2)dt>ct E(u,\dot{u},n,\dot{n}),}
 for all $t>0$.

 First, we prove \eqref{2norm bdd}. Otherwise there exists a sequence $t_k$, $k\in\mathbb{N}$ such that
 \EQ{
 \|u(t_k)\|_{L^2}^2>k\|\nabla u(t_k)\|_{L^2}^2+\varepsilon ( \|\dot{u}(t_k)\|_{L^2}^2+ \frac{1}{2}\|n(t_k)\|_{L^2}^2+\frac{1}{2\alpha^2}\|\dot{n}(t_k)\|_{\dot{H}^{-1}}^2).}
  Since $u(t)$ is $L^2$ bounded, it follows $$\|\nabla u(t_k)\|_{L^2}^2\rightarrow0.$$ By the precompactness of $\{u(t):t\in \R\}$, we
get that up to a subsequence $(u(t_k),n(t_k))$ converges to some $(f,g)$ in $H^1\times L^2$, so the above implies that $u(t_k)\rightarrow0$ strongly in $H^1$. Then by the above inequality, we have
$$ \|\dot{u}(t_k)\|_{L^2}^2+\|n(t_k)\|_{L^2}^2+ \|\dot{n}(t_k)\|_{\dot{H}^{-1}}^2\rightarrow 0$$
which contradicts to the energy conservation and \eqref{eq:Eequiv}.

Next, we prove \eqref{average+}. Applying \eqref{2norm bdd} with $\varepsilon=1/4$, we have
\EQ{&\partial_t\langle u(t),
\dot{u}(t)\rangle_{L^2}=\|\dot{u}\|_{L^2}^2-\|\nabla
u\|_{L^2}^2
-\| u\|_{L^2}^2+\int_{\R^3}nu^2dx\\
=&\|\dot{u}\|_{L^2}^2-\|\nabla
u\|_{L^2}^2
-\| u\|_{L^2}^2-\frac{1}{2}\|n-u^2\|_{L^2}^2+\frac{1}{2}\|n\|_{L^2}^2+\frac{1}{2}\|u\|_{L^4}^4\\
\geq&\frac{1}{2}\|\dot{u}\|_{L^2}^2
+\| u\|_{L^2}^2+\frac{1}{4}\norm{n}_{L^2}^2-C\|\nabla
u\|_{L^2}^2-\frac{1}{4\alpha^2}\|\dot{n}\|_{\dot{H}^{-1}}^2-\frac{1}{2}\|n-u^2\|_{L^2}^2
 .}
Thus,
\EQ{&\int_0^t(C\|\nabla
u\|_{L^2}^2+\frac{C}{4\alpha^2}\|\dot{n}(t)\|_{\dot{H}^{-1}}^2+\frac{1}{2}\|n-u^2\|_{L^2}^2)dt\\
\geq &\int_0^t(\frac{1}{2}\|\dot{u}\|_{L^2}^2
+\| u\|_{L^2}^2+\frac{1}{4}\norm{n}_{L^2}^2+\frac{C-1}{4\alpha^2}\|\dot{n}\|_{\dot{H}^{-1}}^2)dt-CE(u,\dot{u},n,\dot{n})\\
\geq &ct E(u,\dot{u},n,\dot{n}).}
By Lemma \ref{lem:sign}, \eqref{average+} was obtained.

{\bf Step 2:} Uniform small tails.

We claim that for any $\e>0$, there
exists $R>0$ such that at any $t\in \R$, we have
\begin{align*}
\int_{|x|\geq R}\big(&|\nabla u|^2+|u|^2+|u|^4+|\dot{u}|^2+|\nabla
D^{-1} u^2|^2\\
&|\ft^{-1}(|\xi|^{-1}|\wh n(\xi)|)|^6+|n|^2
+|D^{-1}\dot{n}|^2+|\nabla D^{-1}n|^2\big)dx<\e.
\end{align*}
Indeed, since $\{(\cU,\cN)(t):t\in \R\}$ is precompact in $H^1\times
L^2$, by Sobolev embedding and the $L^p$-boundedness of
$D^{-1}\nabla$, we get the result. Then the claim follows immediately.

{\bf Step 3:} Contradiction to the local virial estimates.

For any $R>1$, let $\psi_R(x)=\psi(x/R)$, where $\psi\in C_0^\I(\R^3)$ is a fixed radial function satisfying $0\le\psi\le 1$, $\p_r\psi\le 0$,
$\psi(x)=1$ for $|x|\le 1$ and $\psi(x)=0$ for $|x|\ge 2$.
We consider the local virial estimates as follows:
\EQ{I_R(t)=-2\langle\psi_R\dot{u}(t),(x\cdot\nabla+\frac{3}{2})u(t)\rangle
-\frac{1}{\alpha^2}\langle\psi_RD^{-1}\dot{n}(t), D^{-1}(x\cdot\nabla+2)n(t)\rangle.}
First, we observe that
\begin{align}\label{eq:I_R bound}
|I_R(t)| \lec R\|\dot{u}\|_{L^2}\| u\|_{H^1}+R\|N\|_{L^2}^2\lec R\ \text{for}\ \forall t>0.
\end{align}
On the other hand,
\EQ{I_R'(t)=&-2\frac{d}{dt}\langle\psi_R\dot{u}(t),(x\cdot\nabla+\frac{3}{2})u(t)\rangle\\
&-\frac{1}{\alpha^2}\frac{d}{dt}\langle\psi_RD^{-1}\dot{n}(t),D^{-1}(x\cdot\nabla+2)n(t)\rangle\\
:=&I+II.}

First, we compute $I$. Using the equation and integration by parts,
\EQ{I=&-2\langle\psi_R(\triangle u-u+nu),(x\cdot\nabla+\frac{3}{2})u \rangle-2\langle\psi_R\dot{u},(x\cdot\nabla+\frac{3}{2})\dot{u}\rangle\\
=&2\langle\nabla\psi_R\cdot\nabla u,(x\cdot\nabla+\frac{3}{2})u \rangle+2\langle\psi_R,|\nabla u|^2\rangle\\
&
-\langle\nabla\psi_R\cdot x,|\nabla u|^2\rangle-\langle\nabla\psi_R\cdot x,|u|^2\rangle\\
&+\langle\nabla\psi_R\cdot x, n u^2\rangle+\langle\psi_R x\cdot\nabla n, u^2\rangle+\langle\nabla\psi_R\cdot x,|\dot{u}|^2\rangle.}

Next, we compute $II$.
\EQ{II=&-\langle\psi_RD^{-1}\triangle (n-u^2),D^{-1}(x\cdot\nabla+2)n \rangle\\
&-\frac{1}{\alpha^2}\langle\psi_RD^{-1}\dot{n},(x\cdot\nabla+1)D^{-1}\dot{n} \rangle\\
=&\langle\nabla\psi_R\cdot\nabla D^{-1}n,(x\cdot\nabla+1)D^{-1}n \rangle+\frac{1}{2}\langle\psi_R,|\nabla D^{-1}n|^2 \rangle\\
&
-\frac{1}{2}\langle x\cdot\nabla\psi_R,|\nabla D^{-1}n|^2 \rangle+\langle\psi_RD^{-1}\triangle u^2,D^{-1}(x\cdot\nabla+2)n \rangle\\
&+\frac{1}{2\alpha^2}\langle\psi_R,| D^{-1}\dot{n}|^2\rangle+\frac{1}{2\alpha^2}\langle x\cdot\nabla\psi_R,| D^{-1}\dot{n}|^2\rangle.}

Thus, we get
\EQ{I_R'(t)=&I+II\\
=&2\|\nabla u\|_{L^2}^2+\frac{1}{2}\norm{n}_{L^2}^2+\frac{1}{2\alpha^2}\norm{\dot{n}}_{\dot{H}^{-1}}^2
-2\int_{\R^d}nu^2dx+T_R(t)\\
=&2K_2(u)+\frac{1}{2\alpha^2}\norm{\dot{n}}_{\dot{H}^{-1}}^2+\frac{1}{2}\norm{n-u^2}_{L^2}^2-\LR{n-u^2,u^2}+T_R(t),}
where
\EQ{T_R(t)=&-2\langle1-\psi_R,|\nabla u|^2\rangle-\frac{1}{2}\langle1-\psi_R,|\nabla D^{-1}n|^2\rangle\\
&-\frac{1}{2\alpha^2}\langle1-\psi_R,| D^{-1}\dot{n}|^2\rangle-2\langle1-\psi_R,\nabla D^{-1}u^2\cdot \nabla D^{-1}n \rangle\\
&+2\langle\nabla\psi_R\cdot\nabla D^{-1}u^2,D^{-1}n \rangle-\langle\nabla\psi_R\cdot x,|\nabla u|^2\rangle\\
&-\langle\nabla\psi_R\cdot x,|u|^2\rangle+\langle\nabla\psi_R\cdot x, n u^2\rangle\\
&
+\langle\nabla\psi_R\cdot x,|\dot{u}|^2\rangle+\langle\nabla\psi_R\cdot\nabla D^{-1}n,(x\cdot\nabla+1)D^{-1}n \rangle\\
&
-\frac{1}{2}\langle x\cdot\nabla\psi_R,|\nabla D^{-1}n|^2 \rangle+\frac{1}{2\alpha^2}\langle x\cdot\nabla\psi_R,| D^{-1}\dot{n}|^2\rangle\\
&-\langle[D, \psi_R](D^{-1}(x\cdot\nabla n)),|u|^2\rangle+2\langle\nabla\psi_R\cdot\nabla u,(x\cdot\nabla+\frac{3}{2})u \rangle
}
and the commutator \EQ{[D, \psi_R]f:=D(\psi_Rf)-\psi_RDf.}
From Lemma \ref{lem:var}, we get that
\EQ{I_R'(t)=&2K_2(u)+\frac{1}{2\alpha^2}\norm{\dot{n}}_{\dot{H}^{-1}}^2+\frac{1}{2}\norm{n-u^2}_{L^2}^2-\LR{n-u^2,u^2}+T_R(t)\\
&\geq 2(1-\frac{2}{\sqrt{6}})K_2(u)+(\frac{1}{2}-\frac{1}{\sqrt{6}})\norm{n-u^2}_{L^2}^2+\frac{1}{2\alpha^2}\norm{\dot{n}}_{\dot{H}^{-1}}^2+T_R(t)}
for $\forall t\geq0$.
Assuming for the moment that $\lim_{R\rightarrow\I}T_R(t)=0$, by \eqref{average+}, we have
\EQ{\int_0^tI_R'(t)dt\geq ct E(u,\dot{u},n,\dot{n})}
for $R$ sufficiently large. Thus we get
\[I_R(t)\geq I_R(0)+ct E(u,\dot{u},n,\dot{n}),\]
which contradicts \eqref{eq:I_R bound} for sufficiently large $t$.

To prove $\lim\limits_{R\to\I}T_R(t)=0$, it remains to handle the
commutator term
\[C_R(t):=\langle[D, \psi_R](D^{-1}(x\cdot\nabla n)),|u|^2\rangle,\]
since all the rest terms in $T_R(t)$ follow immediately from step 2
and H\"{o}lder inequality. Note that \EQ{[D,
\psi_R](D^{-1}(x\cdot\nabla n))=[D, \psi_R](x\cdot \nabla
D^{-1}n)-[D, \psi_R](D^{-1}n).} Then this term follows from the
lemma below by taking $f=D^{-1}n$.
\end{proof}

\begin{lem} Assume $0<\epsilon<1$ and $R\ges 1$. Then
\begin{align}
\norm{[D,\psi_R]f}_{L^2}\les& \norm{g}_{L^6(|x|\ges R^{1-\epsilon})}+R^{-\epsilon}\norm{Df}_{L^2},\label{eq:v1}\\
\norm{[D,\psi_R]x\cdot \nabla f}_{L^2}\les& \norm{D
f}_{L^2}, \label{eq:vier2}\\
\norm{[D,\psi_R]x\cdot \nabla f}_{L^2(|x|\les R^{1-\epsilon})}\les&
R^{-\epsilon/2}\norm{D f}_{L^2}+\norm{g}_{L^6(|x|\ges
R^{1-\epsilon})}, \label{eq:vier3}
\end{align}
where $g=\ft^{-1}|\hat f|$.
\end{lem}

\begin{proof}
First we show \eqref{eq:v1}. Taking Fourier transform, we have
\begin{align*}
&|\ft([D,\psi_R]f)(\xi)|\\
\les&
|\int_{\xi=\xi_1+\xi_2}(|\xi_1+\xi_2|-|\xi_2|)\wh{\psi_R}(\xi_1)\wh{f}(\xi_2)|\\
\les&\int_{|\xi_1|\ges |\xi_2|} |\xi_1|
|\wh{\psi_R}(\xi_1)|\cdot|\wh{f}(\xi_2)|+\int_{|\xi_1|\ll |\xi_2|}
|\xi_1|
|\wh{\psi_R}(\xi_1)|\cdot|\wh{f}(\xi_2)|\\
\les&\int_{\xi=\xi_1+\xi_2} |\xi_1|
|\wh{\psi_R}(\xi_1)|\cdot|\wh{f}(\xi_2)|.
\end{align*}
Then
\begin{align*}
\norm{[D,\psi_R]f}_2\les& \norm{\ft^{-1}(|\xi_1|
|\wh{\psi_R}(\xi_1)|)\cdot g}_2\\
\les& \norm{\ft^{-1}(|\xi_1|
|\wh{\psi_R}(\xi_1)|)}_3\norm{g}_{L^6(|x|\ges
R^{1-\epsilon})}+\norm{\ft^{-1}(|\xi_1|
|\wh{\psi_R}(\xi_1)|)}_{L^3(|x|\les R^{1-\epsilon})}\norm{g}_{L^6}\\
\les& \norm{\ft^{-1}(|\xi_1|
|\wh{\psi_R}(\xi_1)|)}_3\norm{g}_{L^6(|x|\ges
R^{1-\epsilon})}+R^{-\epsilon}\norm{Df}_{L^2}
\end{align*}
where we used $|\ft^{-1}(|\xi_1| |\wh{\psi_R}(\xi_1)|)|\les R^{-1}$
and embedding.

For \eqref{eq:vier2} and \eqref{eq:vier3}, direct computations show
that
\begin{align*}
&\ft([D ,\psi_R]x\cdot\nabla f)(\xi) \\
=&-
\int(|\xi| -|\xi_2| )\wh{\psi_R}(\xi-\xi_2)\nabla_{\xi_2}\cdot(\xi_2\wh{f}(\xi_2))d\xi_2\\
=&\int\nabla_{\xi_2}[(|\xi| -|\xi_2| )\wh{\psi_R}(\xi-\xi_2)]\cdot \xi_2\wh{f}(\xi_2)d\xi_2\\
=&\int-|\xi_2|\wh{\psi_R}(\xi-\xi_2)\wh{f}(\xi_2)d\xi_2+i\int(|\xi|
-|\xi_2| )\wh{x\psi_R}(\xi-\xi_2)\cdot
\xi_2\wh{f}(\xi_2)d\xi_2\\
=&\int_{\xi=\xi_1+\xi_2}-|\xi_2|\wh{\psi_R}(\xi_1)\wh{f}(\xi_2)+i(|\xi_1+\xi_2|
-|\xi_2| )\wh{x\psi_R}(\xi_1)\cdot \xi_2\wh{f}(\xi_2).
\end{align*}
 Thus we get
\begin{align*}
|\ft([D^\alpha,\psi_R]x\cdot\nabla
f)(\xi)|\les&\int_{\xi=\xi_1+\xi_2}
|\xi_2|(|\wh{\psi_R}(\xi_1)|+|\wh{x\psi_R}(\xi_1)|\cdot|\xi_1|)\cdot|\wh{f}(\xi_2)|
\end{align*}
and then
\begin{align*}
\norm{[D,\psi_R]x\cdot\nabla f}_{L^2}\les& \norm{D f}_{L^2}.
\end{align*}

Then we have
\begin{align*}
&\ft([D ,\psi_R]x\cdot\nabla f)(\xi)\\
=&\int_{\xi=\xi_1+\xi_2,|\xi_1|\ll
|\xi_2|}-|\xi_2|\wh{\psi_R}(\xi_1)\wh{f}(\xi_2)+i(|\xi_1+\xi_2|
-|\xi_2| )\wh{x\psi_R}(\xi_1)\cdot
\xi_2\wh{f}(\xi_2)\\
&+\int_{\xi=\xi_1+\xi_2,|\xi_1|\ges
|\xi_2|}-|\xi_2|\wh{\psi_R}(\xi_1)\wh{f}(\xi_2)+i(|\xi_1+\xi_2|
-|\xi_2| )\wh{x\psi_R}(\xi_1)\cdot
\xi_2\wh{f}(\xi_2) \\
:=&\ft[M(f)]+\ft[E(f)].
\end{align*}
As before, we have
\begin{align*}
|\ft[E(f)](\xi)|\les&\int_{\xi=\xi_1+\xi_2}
|\xi_1|(|\wh{\psi_R}(\xi_1)|+|\wh{x\psi_R}(\xi_1)|\cdot|\xi_1|)\cdot|\wh{f}(\xi_2)|
\end{align*}
and then
\[\norm{Ef}_{2}\les \norm{g}_{L^6(|x|\ges R^{1-\epsilon})}+R^{-\epsilon}\norm{Df}_{L^2},\label{eq:vier1}\\
.\]

To estimate $M(f)$, we need to exploit a cancelation. Since
\[|\xi_1+\xi_2| -|\xi_2|=\frac{(\xi_1+2\xi_2)}{|\xi_1+\xi_2|+|\xi_2|}\cdot \xi_1\]
we get
\begin{align}
\ft[M(f)]=&\int_{|\xi_1|\ll
|\xi_2|}-|\xi_2|\wh{\psi_R}(\xi_1)\wh{f}(\xi_2)\nonumber\\
&+\int_{|\xi_1|\ll
|\xi_2|}\frac{(\xi_1+2\xi_2)}{|\xi_1+\xi_2|+|\xi_2|}\cdot
i\xi_1\wh{x\psi_R}(\xi_1)\cdot \xi_2\wh{f}(\xi_2)\label{eq:Mf}
\end{align}
Denote $\xi_s=(\xi_{s,1},\xi_{s,2},\xi_{s,3}), s=1,2$, then the second
term equals to
\begin{align*}
&\int_{|\xi_1|\ll
|\xi_2|}\sum_{j,k=1}^3\frac{(\xi_{1,k}+2\xi_{2,k})}{|\xi_1+\xi_2|+|\xi_2|}\cdot
i\xi_{1,k}\wh{x_j\psi_R}(\xi_1)\cdot \xi_{2,j}\wh{f}(\xi_2)\\
=&\int_{|\xi_1|\ll
|\xi_2|}\frac{(\xi_1+2\xi_2)}{|\xi_1+\xi_2|+|\xi_2|}\cdot
\xi_2\wh{\psi_R}(\xi_1)\wh{f}(\xi_2)\\
&+\int_{|\xi_1|\ll
|\xi_2|}\frac{(\xi_1+2\xi_2)}{|\xi_1+\xi_2|+|\xi_2|}\cdot
 \wh{x\otimes \nabla {\psi}_R}(\xi_1)\cdot \xi_2\wh{f}(\xi_2).
\end{align*}
Thus, we get
\begin{align*}
\ft[M(f)]=&\int_{|\xi_1|\ll
|\xi_2|}(\frac{(\xi_1+2\xi_2)}{|\xi_1+\xi_2|+|\xi_2|}\cdot
\xi_2-|\xi_2|)\wh{\psi_R}(\xi_1)\wh{f}(\xi_2)\\
&+\int_{|\xi_1|\ll
|\xi_2|}i\frac{(\xi_1+2\xi_2)}{|\xi_1+\xi_2|+|\xi_2|}\cdot
\frac{\xi_2}{|\xi_2|} \wh{\tilde{\psi}_R}(\xi_1)\wh{D
f}(\xi_2)\\
=&\ft[I]+\ft[II].
\end{align*}

For I, by mean value formula, we have
\[|I|\les \int_{|\xi_1|\ll
|\xi_2|}|\xi_1|\cdot|\wh{\psi_R}(\xi_1)|\cdot|\wh{f}(\xi_2)|\] and
then
\[\norm{I}_{2}\les \norm{g}_{L^6(|x|\ges R^{1-\epsilon})}+R^{-\epsilon}\norm{Df}_{L^2}.\]

For II, we see
\[II=\int K(x-y_1,x-y_2)\tilde{\psi}_R(y_1)D  f(y_2)dy_1dy_2\]
where $K$ is the kernel for the bilinear multiplier
\[K(x,y)=\int e^{i(x\xi_1+y\xi_2)}m(\xi_1,\xi_2)d\xi_1d\xi_2\]
with the symbol
\[m(\xi_1,\xi_2)=\frac{(\xi_1+2\xi_2)}{|\xi_1+\xi_2|+|\xi_2|}\cdot
\frac{\xi_2}{|\xi_2|} \cdot 1_{|\xi_1|\ll |\xi_2|}.\] It is easy to
see from direct computations that $m$ satisfy the Coifman-Meyer's
H\"ormander-type condition, and then
\[|K(x-y_1,x-y_2)|\les (|x-y_1|+|x-y_2|)^{-6}.\]
If $|y_1|\sim R, |x|\les R^{1-\epsilon}$, then $|K(x-y_1,x-y_2)|\les
(R+|y_2|)^{-6}$. Thus we get
\[\norm{II}_{L^2(|x|\les R^{1-\epsilon})}\les R^{-\epsilon/2}\norm{D  f}_2.\]
Therefore, the lemma is proved.
\end{proof}

\subsection*{Acknowledgment}
Z. Guo is supported in part by NNSF of China (NO. 11001003, NO.11271023).
S. Wang is supported by China Scholarship Council and also in part by NNSF of China (NO. 11001003, NO.11271023).

\end{document}